\documentclass[11pt]{article}
\usepackage{amssymb,amsmath,accents}
\usepackage{amscd}
\usepackage{amsfonts,amsthm,mathrsfs}
\usepackage{setspace} 
\usepackage[dvipdfmx]{graphicx}
\usepackage{here}
\usepackage{latexsym}
\usepackage{color}
\usepackage{authblk}

\newtheorem{theorem}{Theorem}
\newtheorem{lemma}[theorem]{Lemma}
\newtheorem{proposition}[theorem]{Proposition}

\theoremstyle{definition}

\newtheorem{remark}[theorem]{Remark}

\title{\textbf{The Helmholtz decomposition of a $BMO$ type vector field in general unbounded domains}}

\author{Yoshikazu Giga \thanks{labgiga@ms.u-tokyo.ac.jp}}
\author{Zhongyang Gu \thanks{zgu@ms.u-tokyo.ac.jp}}
\affil{Graduate School of Mathematical Sciences, The University of Tokyo, 3-8-1 Komaba, Meguro-ku, 153-8914, Tokyo, Japan}

\begin{document}
\date{}

\maketitle
\begin{abstract}
We consider a space of $L^2$ vector fields with bounded mean oscillation whose ``normal'' component to the boundary is well-controlled.
In the case when the dimension $n \geq 3$, we establish its Helmholtz decomposition for arbitrary uniformly $C^3$ domain in $\mathbf{R}^n$.
\end{abstract}
%
\begin{center}
\end{center}

%
\section{Introduction} 
\label{sec:intro}
Let $\Omega \subset \mathbf{R}^n$ be an open subset. 
Throughout this entire paper, regarding dimension $n$ we always consider the case where $n \geq 2$ unless otherwise specified.
The Helmholtz decomposition of a vector field $f = (f_1, f_2, ... , f_n)$ defined in $\Omega$ is formally of the form
\[
f = f_0 + \nabla p,
\]
where $f_0$ is a divergence free vector field satisfying supplemental conditions like boundary condition and $\nabla p$ denotes the gradient of some scalar function $p$.
For $1<r<\infty$, the Helmholtz decomposition of $L^r(\Omega)^n$ is well studied.
To be specific, it gives a topological direct sum decomposition of the vector space $L^r(\Omega)^n$ of all $L^r$-vector fields of the form
\[
L^r(\Omega)^n = L_\sigma^r(\Omega) \oplus G^r(\Omega)
\]
with a space of divergence free vector fields
\[
L_\sigma^r(\Omega) := \overline{C_{\mathrm{c}, \sigma}^\infty(\Omega)}^{\| \cdot \|_{L^r(\Omega)}}
\]
and a space of gradients fields
\[
G^r(\Omega) := \{ \nabla h \in L^r(\Omega)^n \bigm| h \in L_{loc}^r(\Omega) \}.
\]
Here, 
\[
C_{\mathrm{c}, \sigma}^\infty(\Omega) := \{ u \in C_{\mathrm{c}}^\infty(\Omega)^n \bigm| \operatorname{div} u = 0 \},
\] 
where $C_{\mathrm{c}}^\infty(\Omega)$ denotes the space of all smooth functions compactly supported in $\Omega$.
We stress that this decomposition is not just algebraic but topological in the sense that projections $f \mapsto f_0$, $f \mapsto \nabla p$ to each corresponding subspace are bounded in $L^r(\Omega)^n$. In other words,
\begin{align} \label{HDL}
\| f_0 \|_{L^r(\Omega)} \leq C(\Omega,r) \| f \|_{L^r(\Omega)}, \quad \| \nabla p \|_{L^r(\Omega)} \leq C(\Omega,r) \| f \|_{L^r(\Omega)}
\end{align}
with some constant $C(\Omega,r)>0$ depending only on $\Omega$ and $r$.
In the case where $r=2$, this decomposition holds for an arbitrary domain as an orthogonal decomposition of the Hilbert space $L^2(\Omega)^n$.
In particular, the best possible choice of the constant $C(\Omega,2)$ in estimate (\ref{HDL}) is $1$.
For $r \in (1,\infty)$, the Helmholtz decomposition holds for various domains with $C^1$ boundary; see e.g. \cite{Gal}.
For a smooth domain, the decomposition holds for a bounded domain (\cite[n = 3]{Sol}, \cite{FM}), an exterior domain (\cite[n = 3]{Sol}, \cite{SiSo}), a layer domain (\cite{Miy94}, \cite{Far}), an infinite cylinder \cite{Far}, an aperture domain \cite{FaSo}, a perturbed half space \cite{SiSo} as well as the half space.
However, it is also known that there exist unbounded smooth domains which do not admit the $L^r$-Helmholtz decomposition; see e.g. \cite{Bogo}, \cite{Gal}, \cite{MaBo}.
Even if the domain $\Omega$ admits the $L^r$-Helmholtz decomposition, the dependence of constant $C(\Omega,r)$ with respect to $\Omega$ is unclear except the case where $r=2$.
In this decomposition, the gradient field $\nabla p$ must formally satisfy the homogeneous Neumann boundary condition $\partial p / \partial \mathbf{n} = 0$, where $\mathbf{n}$ denotes an exterior unit normal vector field of $\partial \Omega$.
By the way, the divergence free part $f_0$ can be further decomposed into
\[
f_0 = H + \operatorname{curl} \omega,
\]
where $H$ is a harmonic vector field at least when $\Omega$ is a three-dimensional bounded domain.
This decomposition is unique if we impose the boundary condition $H \cdot \mathbf{n} = 0$.
It is often called the $L^r$-Helmholtz-Weyl decomposition and projection operators to each corresponding subspace are bounded for a bounded domain in $\mathbf{R}^n$ ($n = 2,3$) as proved in \cite{KY}.
Recently, this decomposition is studied extensively. The $L^r$-Helmholtz-Weyl decomposition is extended to a three-dimensional exterior domain by \cite{HKSSY21b}, \cite{HKSSY22} and a two-dimensional exterior domain by \cite{HKSSY20}, \cite{HKSSY21a}. In a three-dimensional bounded domain, the space of harmonic vector fields represents the $L^r$-de Rham cohomology group (of degree $2$) if one identifies the space of vector fields with the space of two-forms on $\Omega$ \cite{KY}.

Let us go back to the Helmholtz decomposition. If we consider the space
\begin{equation*} 
	\widetilde{L}^r(\Omega) := \left \{
\begin{array}{lcr}
	L^r(\Omega)^n \cap L^2(\Omega)^n, \quad 2 \leq r < \infty, \\
	L^r(\Omega)^n + L^2(\Omega)^n, \quad 1<r<2
\end{array}
	\right.
\end{equation*}
instead, then for any uniformly $C^1$ domain $\Omega$, the Helmholtz decomposition of $\widetilde{L}^r(\Omega)$ holds for all $r \in (1,\infty)$. We have that
\[
\widetilde{L}^r(\Omega) = \widetilde{L}_\sigma^r(\Omega) \oplus \widetilde{G}^r(\Omega)
\]
where
\begin{equation*} 
	\widetilde{L}_\sigma^r(\Omega) := \left \{
\begin{array}{lcr}
	L_\sigma^r(\Omega) \cap L_\sigma^2(\Omega), \quad 2 \leq r < \infty, \\
	L_\sigma^r(\Omega) + L_\sigma^2(\Omega), \quad 1<r<2
\end{array}
	\right.
\end{equation*}
and
\begin{equation*} 
	\widetilde{G}^r(\Omega) := \left \{
\begin{array}{lcr}
	G^r(\Omega) \cap G^2(\Omega), \quad 2 \leq r < \infty, \\
	G^r(\Omega) + G^2(\Omega), \quad 1<r<2.
\end{array}
	\right.
\end{equation*}
Moreover, there exists a constant $C>0$ such that for any $f \in \widetilde{L}^r(\Omega)$, the Helmholtz decomposition of $f$ satisfies the estimate
\begin{align} \label{HDLT}
\| f_0 \|_{\widetilde{L}^r(\Omega)} + \| \nabla p \|_{\widetilde{L}^r(\Omega)} \leq C \| f \|_{\widetilde{L}^r(\Omega)},
\end{align}
see \cite{FKS}, \cite{FKSH}. We let $\Gamma = \partial \Omega := \overline{\Omega} \setminus \Omega$ to denote the boundary of $\Omega$. In this case, we know that the constant $C$ in estimate (\ref{HDLT}) depends only on $r$ and $\alpha, \beta, K$ where $\alpha, \beta, K$ are parameters that characterize the $C^1$ regularity of $\Gamma$. Here we would like to direct the readers to Section \ref{sub:LT} for the precise definition of a uniformly $C^k$ domain of type $(\alpha, \beta, K)$.

It is impossible to consider the Helmholtz decomposition for $L^\infty$ even if $\Omega=\mathbf{R}^n$ since the projection $f \mapsto \nabla p$ is a composite of the Riesz operators which is not bounded in $L^\infty$.
We have to replace $L^\infty$ with a class of functions with bounded mean oscillation ($BMO$ for short).
In order to understand previous results on the Helmholtz decomposition for $BMO$ vector fields defined in a domain, let us recall some definitions.
We firstly recall the $BMO^\mu$-seminorm for $\mu\in(0,\infty]$.
For a locally integrable function $u$, i.e., $u \in L^1_{loc}(\Omega)$, we define
\[
[u]_{BMO^\mu(\Omega)} := \sup \left\{ \frac{1}{\left|B_r(x)\right|} \int_{B_r(x)} \left| u(y) - u_{B_r(x)} \right| \, dy \biggm| B_r(x) \subset \Omega,\ r < \mu \right\},
\]
where $u_{B_r(x)}$ denotes the average over $B_r(x)$, i.e.,
\[
u_{B_r(x)} := \frac{1}{|B_r(x)|} \int_{B_r(x)} u(y) \, dy
\]
and $B_r(x)$ denotes the open ball of radius $r$ centered at $x$ and $|B_r(x)|$ denotes the Lebesgue measure of $B_r(x)$.
The space $BMO^\mu(\Omega)$ is defined as
\[
BMO^\mu(\Omega) := \left\{ u \in L^1_{loc}(\Omega) \bigm| [u]_{BMO^\mu(\Omega)} < \infty \right\}.
\]
This space may not agree with the space of restrictions of $f \in BMO^\mu(\mathbf{R}^n)$ in $\Omega$.
As in \cite{BG}, \cite{BGMST}, \cite{BGS}, \cite{BGST}, we introduce a seminorm which controls the boundary behavior.
For $\nu \in (0,\infty]$ and $M \subseteq \Gamma$, we set
\[
[u]_{b^\nu(M)} := \sup \left\{ r^{-n} \int_{\Omega\cap B_r(x)} \left| u(y) \right| \, dy \biggm| x \in M, \ 0<r<\nu \right\}.
\]
In these papers, the $BMO$ space 
\[
BMO^{\mu,\nu}_b(\Omega) := \left\{ u \in BMO^\mu(\Omega) \bigm| [u]_{b^\nu(\Gamma)} < \infty \right\}
\]
is considered.
Unfortunately, it turns out that such a boundary control for all components of a vector field is too strict to have the Helmholtz decomposition. 
Hence, we introduced a space of $BMO$ vector fields which separate tangential and normal components.
For $x \in \Omega$, let $d_\Gamma(x)$ denote the distance from the boundary $\Gamma$, i.e., $d_\Gamma(x) := \inf \left\{ |x - y|,\ y \in \Gamma \right\}$.
We consider
\[
vBMO^{\mu,\nu}(\Omega) := \left\{ f \in BMO^\mu(\Omega)^n \bigm| [\nabla d_\Gamma \cdot f]_{b^\nu(\Gamma)} < \infty \right\},
\]
where $\cdot$ denotes the standard inner product in $\mathbf{R}^n$.
The quantity $(\nabla d_\Gamma \cdot f)\nabla d_\Gamma$ on $\Gamma$ is the component of $f$ that is normal to the boundary $\Gamma$.
We set
\[
[f]_{vBMO^{\mu,\nu}(\Omega)} := [f]_{BMO^\mu(\Omega)} + [\nabla d_\Gamma \cdot f]_{b^\nu(\Gamma)}.
\]

If $\Omega = \mathbf{R}^n$ is the whole space, the Helmholtz decomposition of $BMO(\mathbf{R}^n)^n$ holds \cite{Miyak} as the Helmholtz projection in this case can be directly constructed by Riesz transformations.
If $\Omega = \mathbf{R}_+^n := \{ (x', x_n) \in \mathbf{R}^n \bigm| x_n > 0 \}$ is the half space, then by considering even extension for the tangential component and odd extension for the normal component for $f \in vBMO^{\infty,\infty}(\mathbf{R}_+^n)$, we can make use of the Helmholtz projection from the whole space case to prove that $vBMO^{\infty,\infty}(\mathbf{R}_+^n)$ admits the Helmholtz decomposition \cite{GigaGuAMSA}.
If $\Omega$ is a bounded $C^3$ domain, as the boundary $\Gamma$ has a fully curved part in the sense of \cite[Definition 7]{GigaGuPA}, we have that $vBMO^{\mu,\nu}(\Omega) = vBMO^{\infty,\infty}(\Omega)$ for any finite $\mu,\nu>0$ and $vBMO^{\infty,\infty}(\Omega) \subset L^1(\Omega)^n$. In this case, we can omit $\mu,\nu$ from the superscript and simply denote $vBMO^{\infty,\infty}(\Omega)$ by $vBMO(\Omega)$ without causing any ambiguity. Moreover, the multiplication by a H$\ddot{\text{o}}$lder continuous function is bounded in $vBMO(\Omega)$ since $vBMO(\Omega) \subset L^1(\Omega)^n$. As a result, the Helmholtz decomposition for $vBMO(\Omega)$ can be established through a potential theoretical approach \cite{GigaGuMA}.
If $\Omega = \mathbf{R}_h^n := \{ (x', x_n) \in \mathbf{R}^n \bigm| x_n > h(x') \}$ is a perturbed $C^3$ half space that has small perturbation, then in the case when the dimension $n \geq 3$, we could establish the Helmholtz decomposition for the space $vBMOL^2(\mathbf{R}_h^n) := vBMO^{\infty,\infty}(\mathbf{R}_h^n) \cap L^2(\mathbf{R}_h^n)^n$, see \cite{GigaGuP}.
The reason why we consider the intersection of $vBMO^{\infty,\infty}(\mathbf{R}_h^n)$ and $L^2(\mathbf{R}_h^n)^n$ for Helmholtz decomposition is because within this intersection, cut-offs by multiplication become valid \cite{Gu} and thus the potential theoretical approach in \cite{GigaGuMA} can be applied.
Here we would like to direct the readers to Section \ref{sub:ER} for the precise definition for a perturbed $C^3$ half space to have small perturbation.

Analogously, for general unbounded uniformly $C^3$ domain $\Omega \subset \mathbf{R}^n$, we consider the intersection space $vBMOL^2(\Omega) := vBMO^{\infty,\infty}(\Omega) \cap L^2(\Omega)^n$ for Helmholtz decomposition.
Let us note that $vBMOL^2(\Omega) = vBMO^{\mu_0,\nu_0}(\Omega) \cap L^2(\Omega)^n$ for any finite $\mu_0,\nu_0>0$ since estimates
\[
\frac{1}{|B_r(x)|} \int_{B_r(x)} | f - f_{B_r(x)} | \, dy \leq \frac{C(n)}{\mu_0^{\frac{n}{2}}} \| f \|_{L^2(\Omega)}, \quad r^{-n} \int_{B_r(z_0) \cap \Omega} |\nabla d_\Gamma \cdot f| \, dy \leq \frac{C(n)}{\nu_0^{\frac{n}{2}}} \| f \|_{L^2(\Omega)}
\]
hold for any $x \in \Omega$, $z_0 \in \Gamma$ and $r \geq \mu_0,\nu_0$. The purpose of this paper is to establish the Helmholtz decomposition for $vBMOL^2(\Omega)$ for arbitrary uniformly $C^3$ domain $\Omega \subset \mathbf{R}^n$ in the case when $n \geq 3$.
The main theorem reads as follows.

\begin{theorem} \label{MT}
Let $\Omega \subset \mathbf{R}^n$ be a uniformly $C^3$ domain of type $(\alpha, \beta, K)$ with $n \geq 3$.
Then for each $f \in vBMOL^2(\Omega)$, there exists a unique decomposition $f = f_0 + \nabla p$ with 
\begin{align*}
f_0 \in vBMOL^2_\sigma(\Omega) &:= vBMO^{\infty,\infty}(\Omega) \cap L^2_\sigma(\Omega), \\
\nabla p \in GvBMOL^2(\Omega) &:= \left\{ \nabla p \in vBMOL^2(\Omega) \bigm| p \in L^1_\mathrm{loc}(\Omega) \right\}
\end{align*}
satisfying the estimate
\begin{align} \label{MTE}
\| f_0 \|_{vBMOL^2(\Omega)} + \| \nabla p \|_{vBMOL^2(\Omega)} \leq C(\alpha, \beta, K) \| f \|_{vBMOL^2(\Omega)},
\end{align}
where $C(\alpha, \beta, K)>0$ is a constant that depends only on $\alpha, \beta, K$.
In particular, the Helmholtz projection $P_{vBMOL^2}$, which is defined by $P_{vBMOL^2}(f) = f_0$, is a bounded linear mapping on $vBMOL^2(\Omega)$ with range $vBMOL^2_\sigma(\Omega)$ and kernel $GvBMOL^2(\Omega)$.
\end{theorem}

From now on, let us assume that $\Omega$ is a uniformly $C^3$ domain of type $(\alpha, \beta, K)$ with reach $R_\ast$. 
Here the exact meaning of reach $R_\ast$ of $\Gamma = \partial \Omega$ is defined in Section \ref{sub:LT}.
The key idea of the potential theoretical approach in \cite{GigaGuMA}, \cite{GigaGuP} is to solve the equation
\[
\Delta p = \operatorname{div} f \; \; \; \text{in} \; \; \; \Omega, \quad \frac{\partial p}{\partial \mathbf{n}} = f \cdot \mathbf{n} \; \; \; \text{on} \; \; \; \Gamma
\]
strongly. Although there exist a linear operator $f \mapsto p_1$ from $vBMOL^2(\Omega)$ to $L^\infty(\Omega)$ and a constant $C(\alpha,\beta,K,R_\ast)>0$ such that $- \Delta p_1 = \operatorname{div} f$ in $\Omega$ and
\[
\| \nabla p_1 \|_{vBMOL^2(\Omega)} \leq C(\alpha,\beta,K,R_\ast) \| f \|_{vBMOL^2(\Omega)},
\]
see \cite[Theorem 2]{GigaGuP}, it is generally hard to solve the Neumann problem
\begin{align} \label{NP}
\Delta u = 0 \; \; \; \text{in} \; \; \; \Omega, \quad \frac{\partial u}{\partial \mathbf{n}} = g \; \; \; \text{on} \; \; \; \Gamma
\end{align}
in the strong sense. In the case where $\Omega$ is bounded \cite{GigaGuMA}, the Neumann problem \eqref{NP} can be solved strongly for $g \in L^\infty(\Gamma)$ since the boundary $\Gamma$ is compact. Whereas in the case that $\Omega = \mathbf{R}_h^n$ $(n \geq 3)$ is a perturbed $C^3$ half space which has small perturbation \cite{GigaGuP}, by decomposing the boundary $\Gamma$ into the straight part and a curved part, we can treat the straight part as part of the boundary of the half space and the curved part as part of the boundary of some bounded domain.
The Neumann problem \eqref{NP} can therefore be solved strongly for $g \in L^\infty(\Gamma) \cap \dot{H}^{-\frac{1}{2}}(\Gamma)$ where $\dot{H}^{-\frac{1}{2}}(\Gamma)$ denotes the dual space of the homogeneous fractional Sobolev space $\dot{H}^{\frac{1}{2}}(\Gamma)$. 
The restriction $n \geq 3$ is mainly due to the fact that $\dot{H}^{\frac{1}{2}}(\Gamma)$ cannot be viewed as a subspace of distributions.
However, for general unbounded domain $\Omega$, the solvability for the Neumann problem \eqref{NP} is unclear. Hence, our strategy to prove Theorem \ref{MT} follows from \cite{FKS}, \cite{FKSH}.

For simplicity of notations, for any domain $\Omega \subseteq \mathbf{R}^n$ and $1<r<\infty$, we denote the intersection space $BMO^\infty(\Omega) \cap L^r(\Omega)$ by $BMOL^r(\Omega)$.
There exists a bounded linear extension for $BMOL^r(\Omega)$ with any $1<r<\infty$ \cite{Gu}, i.e., there exists a constant $C(\alpha,\beta,K)>0$ such that for any $u \in BMOL^r(\Omega)$, there exists $\overline{u} \in BMOL^r(\mathbf{R}^n)$ such that $\overline{u} \bigm|_{\Omega} = u$ and 
\[
\| \overline{u} \|_{BMOL^r(\mathbf{R}^n)} \leq C(\alpha,\beta,K) \| u \|_{BMOL^r(\Omega)}.
\]
Together with the $BMO-L^r$ interpolation inequality in the whole space $\mathbf{R}^n$ \cite{KW}, we can establish a $BMO-L^r$ interpolation inequality in domain $\Omega$ which guarantees that if $f \in vBMOL^2(\Omega)$, then $f \in \widetilde{L}^q(\Omega)$ for any $2 \leq q < \infty$.
We consider the case where $q = 2n$.
Since $\widetilde{L}^{2n}(\Omega)$ admits the Helmholtz decomposition, there exist unique $f_0 \in \widetilde{L}_\sigma^{2n}(\Omega)$ and $\nabla p \in \widetilde{G}^{2n}(\Omega)$ such that $f = f_0 + \nabla p$.
In the case when the dimension $n \geq 3$, we prove that this decomposition of $f$ is indeed the Helmholtz decomposition of $f$ in $vBMOL^2(\Omega)$.

Here is how we implement the strategy from \cite{FKS}, \cite{FKSH}.
Let $\varepsilon>0$ be sufficiently small. 
We require $\varepsilon$ to be smaller than a specific constant that depends on $\alpha, \beta, K$ which will be determined later in Chapter \ref{sec:3}.
In order to prove that $f_0, \nabla p \in vBMOL^2(\Omega)$, we consider two types of cut-off functions. 
For $x \in \Omega$ such that $d_\Gamma(x) \geq 3 \varepsilon$, we firstly consider $\varphi_x \in C_{\mathrm{c}}^\infty\big( B_{2 \varepsilon}(x) \big)$ with $\varphi_x = 1$ in $B_\varepsilon(x)$ and $\operatorname{supp} \varphi_x \subseteq \overline{B_{\frac{3 \varepsilon}{2}}(x)}$. Let $B_{2n}^x : L_0^{2n}\big( B_{2 \varepsilon}(x) \big) \to W^{1,2n}\big( B_{2 \varepsilon}(x) \big)^n$ be the Bogovski$\breve{\i}$ operator which solves the divergence problem
\[
\operatorname{div} u = g, \quad u \bigm|_{\partial B_{2 \varepsilon}(x)} = 0
\]
for $g \in L^{2n}\big( B_{2 \varepsilon}(x) \big)$ with integral zero in $B_{2 \varepsilon}(x)$. We set $\omega_x := B_{2n}^x\big( \nabla \varphi_x \cdot f_0 \big)$. Note that $\varphi_x f = \varphi_x f_0 + \varphi_x \nabla p$ implies that
\begin{align} \label{HDLBx}
\varphi_x f + \nabla \varphi_x (p - M_x) - \omega_x = \big( \varphi_x f_0 - \omega_x \big) + \nabla \big( \varphi_x (p - M_x) \big)
\end{align}
where $M_x$ denotes the average value of $p$ in $B_{2 \varepsilon}(x)$. We next show that $\varphi_x f + \nabla \varphi_x (p - M_x) - \omega_x$ belongs to $vBMOL^2\big( B_{2 \varepsilon}(x) \big)$. Since $B_{2 \varepsilon}(x)$ is certainly a smooth bounded domain that admits the Helmholtz decomposition \cite{GigaGuMA}, equality (\ref{HDLBx}) is indeed the Helmholtz decomposition of $\varphi_x f + \nabla \varphi_x (p - M_x) - \omega_x$ in $vBMOL^2\big( B_{2 \varepsilon}(x) \big)$.
Hence, there is a constant $C>0$, independent of $x$, such that the estimate
\begin{align*}
&\| \varphi_x f_0 - \omega_x \|_{vBMOL^2\big( B_{2 \varepsilon}(x) \big)} + \| \nabla \big( \varphi_x (p - M_x) \big) \|_{vBMOL^2\big( B_{2 \varepsilon}(x) \big)} \\
&\ \ \leq C \| \varphi_x f + \nabla \varphi_x (p - M_x) - \omega_x \|_{vBMOL^2\big( B_{2 \varepsilon}(x) \big)}
\end{align*}
holds for any $x \in \Omega$ with $d_\Gamma(x) \geq 3 \varepsilon$. 
By using Morrey's inequality and the boundedness of the Bogovski$\breve{\i}$ operator $B_{2n}^x$, we can estimate $\| \omega_x \|_{vBMOL^2\big( B_{2 \varepsilon}(x) \big)}$ by $\| f_0 \|_{L^{2n}(\Omega)}$, which can further be controlled by $\| f \|_{vBMOL^2(\Omega)}$.
Since the Sobolev space $W^{1,n}$ is continuously emdedded into $BMO$, by the Poincar$\acute{\text{e}}$ inequality we can estimate $\| \nabla \varphi_x (p - M_x) \|_{vBMOL^2\big( B_{2 \varepsilon}(x) \big)}$ by $\| \nabla p \|_{\widetilde{L}^{2n}\big( B_{2 \varepsilon}(x) \big)}$, which can also be further controlled by $\| f \|_{vBMOL^2(\Omega)}$. Therefore, the interior $BMO$ estimate for $f_0, \nabla p$ reads as
\[
[f_0]_{BMO^\varepsilon\big( \Omega_{2 \varepsilon} \big)} + [\nabla p]_{BMO^\varepsilon\big( \Omega_{2 \varepsilon} \big)} \leq C(\alpha, \beta, K) \| f \|_{vBMOL^2(\Omega)}
\]
where $\Omega_{2 \varepsilon} := \{ x \in \Omega \bigm| d_\Gamma(x) > 2 \varepsilon \}$.

The second type of cut-off function we consider is supported within a small neighborhood of the boundary.
For $z_0 \in \Gamma$ and $\rho>0$ sufficiently small, we let the intersection between $\Omega$ and the open ball $B_\rho(z_0)$ be denoted by $B^\Omega_\rho(z_0)$ and we consider an open neighborhood of $z_0$ named $U_\rho(z_0)$ that is of a special form which will be explicitly defined in Section \ref{sub:LT}; see Figure \ref{Fnb}.
\begin{figure}[htb]
\centering
\includegraphics[width=7.5cm]{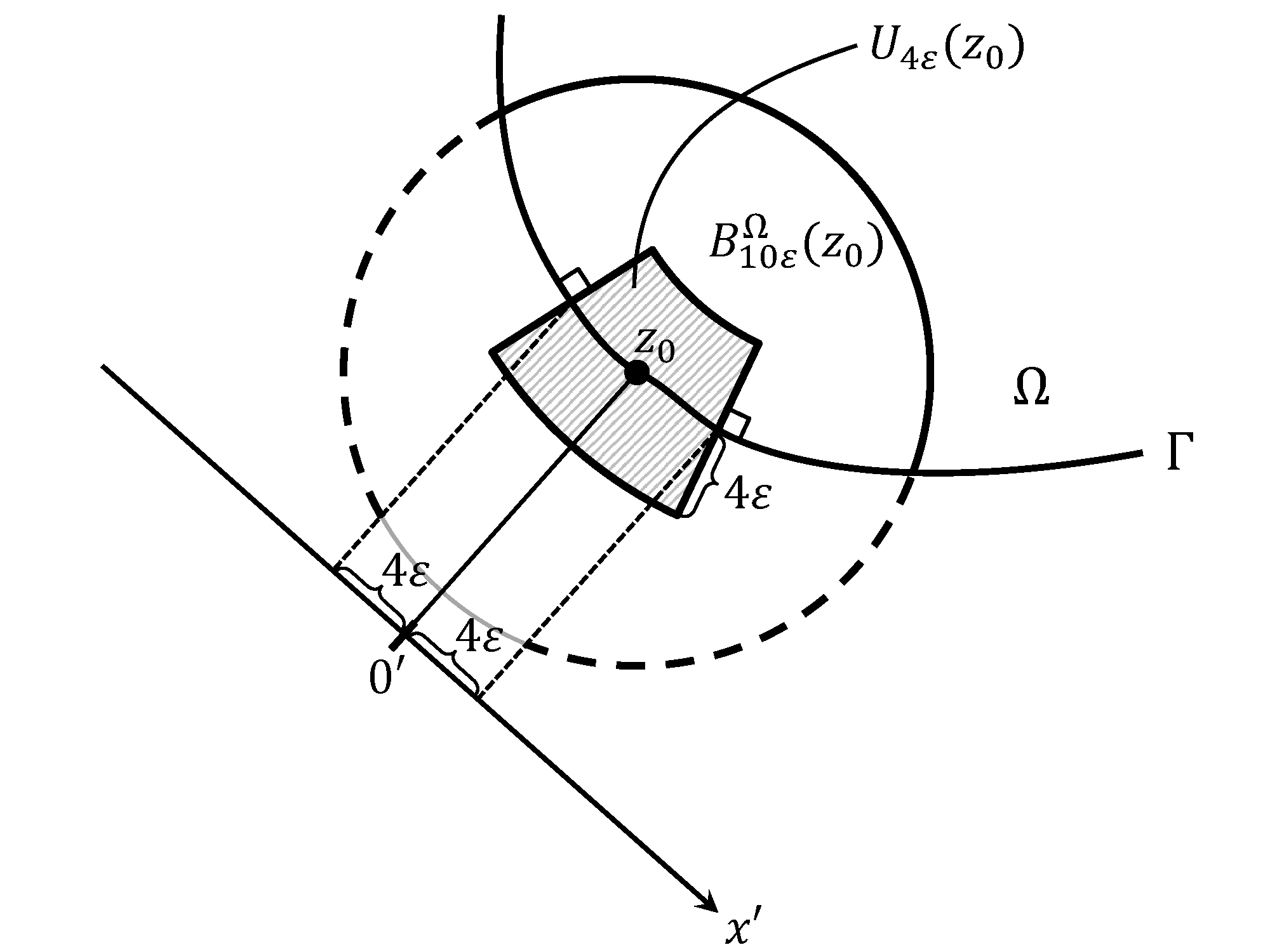}
\caption{$U_{4\varepsilon}(z_0)$ and $B_{10\varepsilon}^\Omega(z_0)$} \label{Fnb}
\end{figure}
We construct $\varphi_{z_0} \in C^2\big( B^\Omega_{12 \varepsilon}(z_0) \big)$ such that $\operatorname{supp} \varphi_{z_0} \subseteq \overline{U_{4 \varepsilon}(z_0)} \subset \overline{B^\Omega_{10 \varepsilon}(z_0)}$, $\varphi_{z_0} = 1$ in $U_{3 \varepsilon}(z_0)$ and $\big( \nabla d_\Gamma \cdot \nabla \varphi_{z_0} \big) (x) = 0$ for any $x \in \Omega$ such that $d_\Gamma(x) < 3 \varepsilon$.
Again, we follow the idea in \cite{FKS}, \cite{FKSH} to set $\omega_{z_0} := B_{2n}^{z_0}\big( \nabla \varphi_{z_0} \cdot f_0 \big)$ where $B_{2n}^{z_0} : L_0^{2n}\big( B^\Omega_{12 \varepsilon}(z_0) \big) \to W^{1,2n}\big( B^\Omega_{12 \varepsilon}(z_0) \big)^n$ is the Bogovski$\breve{\i}$ operator which solves the divergence problem 
\[
\operatorname{div} u = g, \quad u \bigm|_{\partial B^\Omega_{12 \varepsilon}(z_0)} = 0
\]
for $g \in L^{2n}\big( B^\Omega_{12 \varepsilon}(z_0) \big)$ with integral zero in $B^\Omega_{12 \varepsilon}(z_0)$. By multiplying the cut-off function $\varphi_{z_0}$ to $f = f_0 + \nabla p$, we have that in this case
\begin{align} \label{HDPHz}
\varphi_{z_0} f + \nabla \varphi_{z_0} (p - M_{z_0}) - \omega_{z_0} = \big( \varphi_{z_0} f_0 - \omega_{z_0} \big) + \nabla \big( \varphi_{z_0} (p - M_{z_0}) \big)
\end{align}
where $M_{z_0}$ denotes the average value of $p$ in $B^\Omega_{12 \varepsilon}(z_0)$.
Different from the interior case where equality (\ref{HDLBx}) is treated as being defined in a ball, in this case we treat equality (\ref{HDPHz}) as being defined in a perturbed $C^3$ half space $\mathbf{R}_{h_{z_0}^\ast}^n$ which has small perturbation, i.e., we view $B^\Omega_{12 \varepsilon}(z_0)$ as being contained in a perturbed $C^3$ half space $\mathbf{R}_{h_{z_0}^\ast}^n$ which has small perturbation.
Such concern is the key idea of \cite{FKS}, \cite{FKSH}.
We then show that $\varphi_{z_0} f + \nabla \varphi_{z_0} (p - M_{z_0}) - \omega_{z_0}$ belongs to $vBMOL^2\big( \mathbf{R}_{h_{z_0}^\ast}^n \big)$, which admits the Helmholtz decomposition \cite{GigaGuP}. As a result, equality (\ref{HDPHz}) is indeed the Helmholtz decomposition of $\varphi_{z_0} f + \nabla \varphi_{z_0} (p - M_{z_0}) - \omega_{z_0}$ in $vBMOL^2\big( \mathbf{R}_{h_{z_0}^\ast}^n \big)$ and there exists a constant $C(\alpha,\beta,K)>0$ such that the estimate
\begin{align*}
&\| \varphi_{z_0} f_0 - \omega_{z_0} \|_{vBMOL^2\big( \mathbf{R}_{h_{z_0}^\ast}^n \big)} + \| \nabla \big( \varphi_{z_0} (p - M_{z_0}) \big) \|_{vBMOL^2\big( \mathbf{R}_{h_{z_0}^\ast}^n \big)} \\
&\ \ \leq C(\alpha,\beta,K) \| \varphi_{z_0} f + \nabla \varphi_{z_0} (p - M_{z_0}) - \omega_{z_0} \|_{vBMOL^2\big( \mathbf{R}_{h_{z_0}^\ast}^n \big)}
\end{align*}
holds for any $z_0 \in \Gamma$. 
Then, we perform similar estimates as in the interior case. 
We use Morrey's inequality and the boundedness of the Bogovski$\breve{\i}$ operator $B_{2n}^{z_0}$ to estimate $\| \omega_{z_0} \|_{vBMOL^2\big( \mathbf{R}_{h_{z_0}^\ast}^n \big)}$ by $\| f \|_{vBMOL^2(\Omega)}$.
Since the Sobolev space $W^{1,n}$ is continuously embedded in $BMO$, we use the Poincar$\acute{\text{e}}$ inequality to estimate $\| \nabla \varphi_{z_0} (p - M_{z_0}) \|_{vBMOL^2\big( \mathbf{R}_{h_{z_0}^\ast}^n \big)}$ by $\| f \|_{vBMOL^2(\Omega)}$. Therefore, the $vBMO$ estimate for $f_0, \nabla p$ up to the boundary reads as
\[
[f_0]_{BMO^{\frac{\varepsilon}{2}}\big( \Gamma_{3 \varepsilon}^{\mathbf{R}^n} \big)} + [\nabla d \cdot f_0]_{b^\varepsilon(\Gamma)} + [\nabla p]_{BMO^{\frac{\varepsilon}{2}}\big( \Gamma_{3 \varepsilon}^{\mathbf{R}^n} \big)} + [\nabla d \cdot \nabla p]_{b^\varepsilon(\Gamma)} \leq C(\alpha,\beta,K) \| f \|_{vBMOL^2(\Omega)}.
\]
Together with the interior $BMO$ estimate for $f_0, \nabla p$, we obtain Theorem \ref{MT}.

As we have discussed before, the $BMO-L^2$ interpolation inequality implies the continuous embedding of $vBMOL^2(\Omega)$ in $\widetilde{L}^{2n}(\Omega)$. As a result, the existence of the decomposition for $f$ in $vBMOL^2(\Omega)$ into the sum of a solenoidal vector field $f_0$ and a gradient field $\nabla p$ is already guaranteed by the Helmholtz decomposition of $f$ in $\widetilde{L}^{2n}(\Omega)$. In order to establish the Helmholtz decomposition for $f \in vBMOL^2(\Omega)$, it is sufficient to establish estimate (\ref{MTE}) with the constant depending only on the regularity of $\Gamma = \partial \Omega$.
Since the $L^2$ part of estimate (\ref{MTE}) is also guaranteed by the Helmholtz decomposition of $f$ in $\widetilde{L}^{2n}(\Omega)$, we only need to estimate the $vBMO^{\mu,\nu}$-seminorm of $f_0$ and $\nabla p$ for finite $\mu,\nu$. Thus, it is natural to follow the strategy from \cite{FKS}, \cite{FKSH} to cut-off $f$ in order to make use of existing results on the Helmholtz decomposition of the space $vBMOL^2$. 
For the interior part that is far away from the boundary $\Gamma$, we could locally make use of the Helmholtz decomposition of $vBMOL^2$ in either an open ball (which is a bounded domain) or the half space. Both choices would ensure that the constant in the estimate is explicitly determined as a fixed number.
For the region that is near the boundary $\Gamma$, although we can still locally make use of the Helmholtz decomposition of $vBMOL^2$ in a bounded domain, but after the cut-off, part of the boundary $\Gamma$ is involved and in this case, we have an estimate with a constant whose dependency on $\Gamma$ is unclear. 
Therefore, for the region that is near $\Gamma$, the only choice we can consider is to locally make use of the Helmholtz decomposition of $vBMOL^2$ in a slightly perturbed half space.
Since the Helmholtz decomposition of $vBMOL^2$ in a slightly perturbed half space is only established in the case when the dimension $n \geq 3$ \cite{GigaGuP}, our main theorem also needs to require the dimension $n \geq 3$.
If the Helmholtz decomposition of $vBMOL^2$ in a slightly perturbed half space can be established in the case when $n=2$, then the statement of Theorem \ref{MT} is also valid for the case $n=2$.

The organization of this paper is as follows.
Main contents in Chapter \ref{sec:2} are to ensure that the boundedness of the Bogovski$\breve{\i}$ operator, Morrey's inequality and the Poincar$\acute{\text{e}}$ inequality holds for both domains $B_\rho(x)$ and $B^\Omega_\rho(z_0)$ with explicit dependency on constants in estimates. We construct our desired cut-off functions $\varphi_x$ and $\varphi_{z_0}$ in Section \ref{sub:LT}. In Section \ref{sub:GPI}, we prove that the domain $B^\Omega_\rho(z_0)$ with sufficiently small $\rho$ is bounded Lipschitz and star-like.
For a perturbed $C^3$ half space $\mathbf{R}_h^n$ that has small perturbation, the constant in the estimate of the Helmholtz decomposition of $vBMOL^2\big( \mathbf{R}_h^n \big)$ depends on both the boundary regularity and the reach of boundary $\partial \mathbf{R}_{h_0}^n$. 
Hence in Section \ref{sub:ER}, we show that the reach of the boundary of a uniformly $C^3$ domain depends on its boundary regularity and the domain $B^\Omega_\rho(z_0)$ can be viewed as being contained in a perturbed $C^3$ half space with small perturbation whose boundary regularity is determined by the regularity of $\Gamma$. In Section \ref{sub:BMP}, we establish necessary inequalities with explicit constant dependency for domains $B_\rho(x)$ and $B^\Omega_\rho(z_0)$. Chapter \ref{sec:3} gives the proof to Theorem \ref{MT}. In Section \ref{sub:IHT}, we reprove an extension theorem for $BMOL^r$ in a uniform domain and use it to establish the domain version $BMO-L^r$ interpolation inequality. In Section \ref{sub:IE} and \ref{sub:EUB}, we establish the interior $BMO$ estimate and the $vBMO$ estimate up to the boundary for $f_0, \nabla p$. In Section \ref{sub:CSoGr}, we give characterizations to intersections of $vBMO^{\infty,\infty}(\Omega)$ with $\widetilde{L}^{2n}_\sigma(\Omega)$ and $\widetilde{G}^{2n}(\Omega)$.

\section{Domain dependency of some basic inequalities} 
\label{sec:2}
\subsection{Localization tool} \label{sub:LT}
Throughout this paper, we denote $x' := (x_1, x_2, ... , x_{n-1})$ for $x \in \mathbf{R}^n$ and $\nabla' := (\partial_1, \partial_2, ... , \partial_{n-1})$.
Let us firstly recall some well-known details about a uniformly $C^k$ domain $\Omega \subset \mathbf{R}^n$ for $k \in \mathbf{N}$, see e.g. \cite[Section I.3.2]{HSo}.
That is to say, there exist constants $\alpha, \beta, K>0$ such that for each $z_0 \in \Gamma$, up to translations and rotations, there exists a function $h_{z_0} \in C^k\big( B_\alpha(0') \big)$, where $B_\alpha(0')$ denotes the open ball $B_\alpha(0')$ in $\mathbf{R}^{n-1}$ of radius $\alpha$ with center $0'$, satisfying the following properties:
\begin{enumerate} 
\item[(i)] 
\[
\sup_{0 \leq s \leq k} \big\| (\nabla')^s h_{z_0} \big\|_{L^\infty\big( B_\alpha(0') \big)} \leq K; \quad \big( \nabla' h_{z_0} \big) (0')=0', \quad h_{z_0}(0')=0,
\]
\item[(i\hspace{-1pt}i)] $\Omega \cap U_{\alpha, \beta, h_{z_0}}(z_0)=\left\{ (x',x_n)\in\mathbf{R}^n \bigm| h_{z_0}(x')<x_n<h_{z_0}(x')+\beta,\ |x'|<\alpha \right\}$ where
\[
	U_{\alpha, \beta, h_{z_0}}(z_0) := \left\{ (x',x_n)\in\mathbf{R}^n \bigm| h_{z_0}(x')-\beta < x_n < h_{z_0}(x')+\beta,\ |x'|<\alpha \right\},
\]
\item[(i\hspace{-1pt}i\hspace{-1pt}i)] $\Gamma\cap U_{\alpha, \beta, h_{z_0}}(z_0)=\left\{ (x',x_n)\in\mathbf{R}^n \bigm| x_n = h_{z_0}(x'),\ |x'|<\alpha \right\}$.
\end{enumerate}
We say that $\Omega$ is of type $(\alpha, \beta, K)$; see Figure \ref{Ftyp}. 
Since requiring $\beta$ to be sufficiently small does not affect values of $\alpha$ and $K$, without loss of generality, within this paper we may assume that $\beta< \mathrm{min} \{ \alpha, \frac{2}{nK} \}$.
\begin{figure}[htb]
\centering
\includegraphics[width=7.5cm]{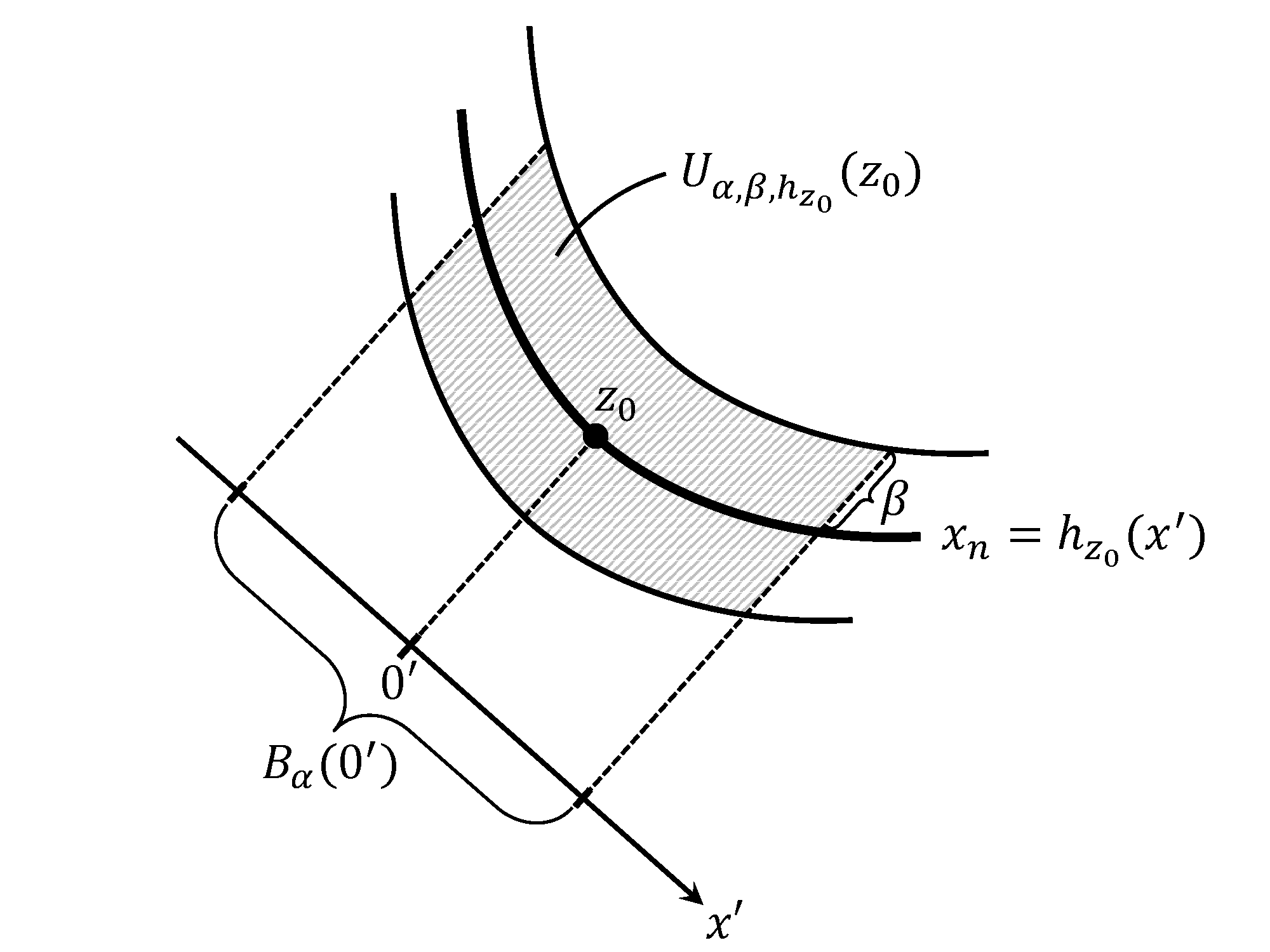}
\caption{$U_{\alpha,\beta,h_{z_0}}(z_0)$} \label{Ftyp}
\end{figure}

Let $d$ denote the signed distance function from $\Gamma$ which is defined by
\begin{equation} 
	d(x) := \left \{
\begin{array}{r}
	\displaystyle \inf_{y\in\Gamma}|x-y| \quad\text{for}\quad x\in\Omega, \\
	\displaystyle -\inf_{y\in\Gamma}|x-y| \quad\text{for}\quad x\notin\Omega 
\end{array}
	\right.
\end{equation}
%
so that $d(x)=d_\Gamma(x)$ for $x\in\Omega$.
For a uniformly $C^k$ domain $\Omega$, there is $R^\Omega>0$ such that for $x \in \Omega$ with $\left|d(x)\right|<R^\Omega$, there is a unique point $\pi x \in \Gamma$ such that $|x-\pi x|=\left|d(x)\right|$.
The supremum of such $R^\Omega$ is called the reach of $\Gamma$ in $\Omega$, we denote this supremum by $R_\ast^\Omega$.
Let $\Omega^{\mathrm{c}}$ be the complement of $\Omega$ in $\mathbf{R}^n$. Similarly, there is $R^{\Omega^{\mathrm{c}}}>0$ such that for $x \in \Omega^{\mathrm{c}}$ with $\left|d(x)\right| < R^{\Omega^{\mathrm{c}}}$, we can also find a unique point $\pi x \in \Gamma$ such that $|x-\pi x|=\left|d(x)\right|$.
The supremum of such $R^{\Omega^{\mathrm{c}}}$ is called the reach of $\Gamma$ in $\Omega^{\mathrm{c}}$, we denote this supremum by $R_\ast^{\Omega^{\mathrm{c}}}$. We then define 
\[
R_* := \min \left( R^\Omega_*, R^{\Omega^{\mathrm{c}}}_* \right),
\]
which we call it the reach of $\Gamma$.
Moreover, $d$ is $C^k$ in the $\rho$-neighborhood of $\Gamma$ for any $\rho \in (0,R_\ast)$, i.e., $d \in C^k\left(\Gamma^{\mathbf{R}^n}_\rho\right)$ for any $\rho \in (0,R_\ast)$ with
\[
\Gamma^{\mathbf{R}^n}_\rho := \left\{ x \in \mathbf{R}^n \bigm| \left| d(x) \right| < \rho \right\};
\]
see e.g. \cite[Chap.\ 14, Appendix]{GT}, \cite[\S 4.4]{KP}.

We next consider $\Omega$ as a uniformly $C^k$ domain with $k \geq 3$.
For $z_0 \in \Gamma$ and $\rho \in (0, R_\ast)$, we set
\[
U_\rho(z_0) := \left\{ x \in \Omega \bigm| (\pi x)' \in B_\rho(0'), d_\Gamma(x) < \rho \right\},
\]
where $\pi x$ denotes the projection of $x$ on $\Gamma$; see Figure \ref{Fnb}.
Let $F_0: V_\rho := B_\rho(0') \times (0,\rho) \to U_\rho(z_0)$
\begin{eqnarray} \label{NCC}
x = F_0(\eta) = \left\{
\begin{array}{lcl}
\eta' + \eta_n \nabla'd ( \eta', h_{z_0}(\eta') ); \\
h_{z_0}(\eta') + \eta_n \partial_{x_n} d ( \eta', h_{z_0}(\eta') )
\end{array}
\right.
\end{eqnarray}
be the normal coordinate change in $U_\rho(z_0)$. For any $\varepsilon \in (0,1)$, there is a constant $c_\varepsilon^\Omega$ which depends on $\Omega$ and $\varepsilon$ only, such that for any $\rho \in (0, c_\varepsilon^\Omega)$ and $z_0 \in \Gamma$, the estimates
\begin{equation} \label{UCG}
\begin{split}
\| \nabla F_0 - I \|_{L^\infty(V_\rho)} &< \varepsilon, \\
\| \nabla F_0^{-1} - I \|_{L^\infty\big( U_\rho(z_0) \big)} &< \varepsilon
\end{split}
\end{equation}
hold simultaneously \cite[Proposition 3]{Gu}.
Therefore, by Hadamard's inverse function theorem, see e.g. \cite[\S 6.2]{KP}, the $C^{k-1}$ mapping $F_0: V_\rho \to U_\rho(z_0)$ is indeed a $C^2$ diffeomorphism. Note that $\nabla F_0^{-1}$ is actually $\frac{1}{\mathrm{det}(\nabla F_0 )} \cdot \mathrm{adj(\nabla F_0)}$ where $\mathrm{adj}(\nabla F_0)$ denotes the adjugate of $\nabla F_0$. By considering the relation
\[
\big( \partial_{x_i} \nabla F_0^{-1} \big)(x) = \sum_{j=1}^n \big( \nabla F_0^{-1} \big)_{ij}(x) \cdot \left\{ \frac{\partial}{\partial \eta_j} \left( \frac{\mathrm{adj(\nabla F_0)} }{\mathrm{det}(\nabla F_0 )} \right) \right\}(\eta),
\]
by estimates (\ref{UCG}) we can deduce that there exists a constant $C=C(K)>0$ such that
\begin{align} \label{C2EN}
\| F_0 \|_{C^2(V_\rho)} + \| F_0^{-1} \|_{C^2\big( U_\rho(z_0) \big)} \leq C.
\end{align}

For $\rho \in (0,R_\ast)$, we set $\Omega_\rho := \Omega \setminus \overline{\Gamma_\rho^{\mathbf{R}^n}}$.
In order to localize the problem, we introduce two types of cut-off functions. one type are supported in a small ball away from the boundary whereas the other type are supported in a small neighborhood of the boundary.

\begin{proposition} \label{CFB}
Let $\Omega \subset \mathbf{R}^n$ be a uniformly $C^1$ domain and $R_\ast$ be the reach of boundary $\Gamma = \partial \Omega$.
Let $\varepsilon \in (0, \frac{R_\ast}{3})$ and $x \in \overline{\Omega_{3 \varepsilon}}$. There exists $\varphi_x \in C_{\mathrm{c}}^\infty(\mathbf{R}^n)$ which satisfies the following conditions:
\begin{itemize}
\item $0 \leq \varphi_x \leq 1$,

\item $\operatorname{supp} \varphi_x \subseteq \overline{B_{\frac{3 \varepsilon}{2}}(x)}$,

\item $\varphi_x = 1$ in $B_\varepsilon(x)$,

\item $\| \varphi_x \|_{C^2(\mathbf{R}^n)} \leq C \varepsilon^{-2}$ where $C$ is a constant independent of $x \in \overline{\Omega_{3 \varepsilon}}$.
\end{itemize}
\end{proposition}
\begin{proof}
Let us recall a well-known construction of a $1$ dimensional cut-off function which is widely used in various contents, see e.g. \cite[Lemma 2.20 and Lemma 2.21]{Lee}. We consider
\begin{align*} 
\phi(t) := \frac{\zeta(\frac{9}{4}-t^2)}{\zeta(t^2-1) + \zeta(\frac{9}{4}-t^2)}, \quad t \in \mathbf{R}
\end{align*}
with
\begin{eqnarray} \label{CFZ}
\zeta(t) :=
\left\{
\begin{array}{lcl}
\mathrm{exp} ( - \frac{1}{t} ) \quad \text{for} \quad t > 0, \\
0 \quad \quad \quad \quad \, \text{for} \quad t \leq 0.
\end{array}
\right.
\end{eqnarray}
In our case, an easy check tells us that $\phi \in C_{\mathrm{c}}^\infty(\mathbf{R})$ which satisfies $\phi(t) = 1$ for $|t| \leq 1$ and $\phi(t) = 0$ for $|t| \geq 3/2$. For $x \in \overline{\Omega_{3 \varepsilon}}$, we let 
\[
\varphi_x(y) := \phi \big( \frac{|y-x|}{\varepsilon} \big), \quad y \in \mathbf{R}^n.
\]
It is obvious that $0 \leq \varphi_x \leq 1$, $\operatorname{supp} \varphi_x \subseteq \overline{B_{\frac{3 \varepsilon}{2}}(x)}$ and $\varphi_x =1$ in $B_\varepsilon(x)$. A direct calculation shows that
\[
\| \varphi_x \|_{C^2(\mathbf{R}^n)} \leq C(n) \| \phi \|_{C^2(\mathbf{R})} \cdot \varepsilon^{-2}.
\]
\end{proof}
\begin{proposition} \label{CFI}
Let $\Omega \subset \mathbf{R}^n$ be a uniformly $C^3$ domain and $R_\ast$ be the reach of boundary $\Gamma = \partial \Omega$.
Let $0<\varepsilon< \mathrm{min} \{ \frac{R_\ast}{5}, \frac{c_{1/2}^\Omega}{5} \}$ and $z_0 \in \Gamma$. There exists $\varphi_{z_0} \in C^2(\Omega)$ which satisfies the following conditions:
\begin{itemize}
\item $0 \leq \varphi_{z_0} \leq 1$,

\item $\operatorname{supp} \varphi_{z_0} \subseteq \overline{U_{4 \varepsilon}(z_0)}$,

\item $\varphi_{z_0} = 1$ in $U_{3 \varepsilon}(z_0)$,

\item $\| \varphi_{z_0} \|_{C^2(\overline{\Omega})} \leq C \varepsilon^{-2}$ where $C$ is a constant independent of $z_0 \in \Gamma$,

\item $\nabla d \cdot \nabla \varphi_{z_0} = 0$ in $\Gamma_{3 \varepsilon}^{\mathbf{R}^n}$.
\end{itemize}
\end{proposition}

\begin{proof}
Similar to the construction of $\phi$ in the proof of Proposition \ref{CFB}, in this case we consider
\begin{align*}
\psi(t) := \frac{\zeta(4-t)}{\zeta(t-3) + \zeta(4-t)}, \quad t \in \mathbf{R}
\end{align*}
where $\zeta$ is the function defined by expression (\ref{CFZ}).
Obviously, $\psi \in C_{\mathrm{c}}^\infty(\mathbf{R})$ which satisfies $\operatorname{supp} \psi \subseteq \overline{B_4(0)}$, $0 \leq \psi \leq 1$ and $\psi = 1$ in $B_3(0)$. 
Let $\mathbf{R}_+^n := \{ (\eta',\eta_n) \in \mathbf{R}^n \bigm| \eta_n > 0 \}$ be the upper half space.
We then define
\[
\varphi_V(\eta',\eta_n) := \psi \bigg( \frac{\eta_n}{\varepsilon} \bigg), \quad \varphi_H(\eta',\eta_n) := \psi \bigg( \frac{|\eta'|}{\varepsilon} \bigg)
\]
for any $(\eta',\eta_n) \in \mathbf{R}_+^n$. Note that for $0< \eta_n, \xi_n < 3 \varepsilon$ with $\eta_n \neq \xi_n$, we have that $\varphi_V(\eta',\eta_n) = \varphi_V(\eta',\xi_n) = 1$ and for any $|\eta'|<3 \varepsilon$, we have that $\varphi_H(\eta',\eta_n) = 1$.
By setting
\[
\varphi_\ast(\eta',\eta_n) := \varphi_V(\eta',\eta_n) \cdot \varphi_H(\eta',\eta_n)
\]
for any $(\eta',\eta_n) \in \mathbf{R}_+^n$, our desired $\varphi_{z_0}$ can be constructed as 
\[
\varphi_{z_0}(x) := \big( \varphi_\ast \circ F_0^{-1} \big)(x), \quad x \in U_{4 \varepsilon}(z_0).
\]
It is obvious that $0 \leq \varphi_{z_0} \leq 1$, $\operatorname{supp} \varphi_{z_0} \subseteq \overline{U_{4 \varepsilon}(z_0)}$ and $\varphi_{z_0} = 1$ in $U_{3 \varepsilon}(z_0)$. Since estimate (\ref{C2EN}) provide a uniform control on $\| F_0^{-1} \|_{C^2\big( U_{4 \varepsilon}(z_0) \big)}$, we have that
\[
\| \varphi_{z_0} \|_{C^2\big( U_{4 \varepsilon}(z_0) \big)} \leq C(K,n) \varepsilon^{-2}.
\]
Since by our construction we have that $\varphi_\ast(\eta',\eta_n) = \varphi_\ast(\eta',\xi_n)$ for any $0<\eta_n, \xi_n < 3 \varepsilon$ such that $\eta_n \neq \xi_n$, the fact that
\[
(\nabla_x d) \circ F_0 \cdot (\nabla_x \varphi_{z_0}) \circ F_0 = \partial_{\eta_n} \big( \varphi_{z_0} \circ F_0 \big);
\]
see e.g. \cite[Proof of Lemma 5 (i)]{GigaGuMA}, would imply that $\nabla d \cdot \nabla \varphi_{z_0} = 0$ in $\Gamma_{3 \varepsilon}^{\mathbf{R}^n}$.
\end{proof}

\subsection{Geometric properties of the intersection of $\Omega$ with a ball centered on $\Gamma$}
\label{sub:GPI}

In this section, let $\Omega \subset \mathbf{R}^n$ be a uniformly $C^2$ domain of type $(\alpha, \beta, K)$ with boundary $\Gamma = \partial \Omega$. Let $R_\ast$ be the reach of $\Gamma$.
For $z_0 \in \Gamma$ and $\rho \in (0,R_\ast)$, we denote the intersection of $\Omega$ with the ball $B_\rho(z_0)$ by $B^\Omega_\rho(z_0)$, i.e., $B^\Omega_\rho(z_0) := B_\rho(z_0) \cap \Omega$.

Let us further recall that a bounded domain $D \subset \mathbf{R}^n$ is said to be a Lipschitz domain if there exist constants $\alpha_0, \beta_0, K_0 > 0$ such that for each $w_0 \in \partial D := \overline{D} \setminus D$, up to translations and rotations, there exists a Lipschitz function $h_{w_0} \in C^{0,1}\big( B_{\alpha_0}(0') \big)$ satisfying
\begin{enumerate} 
\item[(i)] 
\[
\| h_{w_0} \|_{C^{0,1}\big( B_{\alpha_0}(0') \big)} \leq K_0, \quad h_{w_0}(0')=0,
\]
\item[(i\hspace{-1pt}i)] $D \cap U_{\alpha_0, \beta_0, h_{w_0}}(w_0)=\left\{ (x',x_n) \in \mathbf{R}^n \bigm| h_{w_0}(x')<x_n<h_{w_0}(x')+\beta_0,\ |x'|<\alpha_0 \right\}$ where
\[
U_{\alpha_0, \beta_0, h_{w_0}}(w_0) := \left\{ (x',x_n) \in \mathbf{R}^n \bigm| h_{w_0}(x')-\beta_0 < x_n < h_{w_0}(x')+\beta_0,\ |x'|<\alpha_0 \right\},
\]
\item[(i\hspace{-1pt}i\hspace{-1pt}i)] $\partial D \cap U_{\alpha_0, \beta_0, h_{w_0}}(w_0)=\left\{ (x',x_n) \in \mathbf{R}^n \bigm| x_n = h_{w_0}(x'),\ |x'|<\alpha_0 \right\}$;
\end{enumerate}
see e.g. \cite[Section I.3.2]{HSo}. We call the constant $K_0$ to be the Lipschitz constant for boundary $\partial D$.

\begin{lemma} \label{BOL}
Let $z_0 \in \Gamma$ and $0<\rho<\mathrm{min} \{(96nK+4)^{-1}, \alpha, \beta, R_\ast\}$. Then $B^\Omega_\rho(z_0)$ is a bounded Lipschitz domain with Lipschitz constant $L=L(K)>0$.
\end{lemma}

Lemma \ref{BOL} is intuitively clear at least qualitatively. However, we give a detailed proof since it seems that there is no explicit proof in the literature.
Let $z_0 \in \Gamma$ and $\rho<\mathrm{min} \, \{ \alpha, \beta, R_\ast\}$.
Before we start to prove anything, we would like to firstly clarify some concepts that are necessary to understand the notations in this proof.
Let $w_0 \in \Gamma \cap \partial B_\rho(z_0)$.
Here we mean that $w_0 = \big( w_0', (w_0)_n \big)$ is the coordinate of the point $w_0$ in the local coordinate system in $z_0$.
Without loss of generality, we assume that the coordinate system in $z_0$ is generated by the base of unit normal vectors $\{e_1, e_2, ... , e_n \}$ where the $j$-th component of $e_i$ is $\delta_{ij}$ for all $1 \leq i,j \leq n$.
Let $(e_n)^\Omega$ be the unit normal vector at $w_0$ with respect to boundary $\Gamma$, i.e., we have that in this case
\begin{align*} 
(e_n)^\Omega = \left( - \frac{\nabla' h_{z_0}(w_0')}{\big( 1 + \big| \nabla' h_{z_0}(w_0') \big|^2 \big)^{\frac{1}{2}}}, \, \frac{1}{\big( 1 + \big| \nabla' h_{z_0}(w_0') \big|^2 \big)^{\frac{1}{2}}} \right).
\end{align*}
Let $(e_n)^B$ be the unit normal vector at $w_0$ with respect to boundary $\partial B_\rho(z_0)$, i.e., we have that in this case
\begin{align*} 
(e_n)^B = - \left( \frac{w_0'}{\rho}, \, \frac{(w_0)_n}{\rho} \right) = - \left( \frac{w_0'}{\rho}, \, \frac{h_{z_0}(w_0')}{\rho} \right).
\end{align*}
We then set 
\[
(e_n)^I := \frac{(e_n)^\Omega + (e_n)^B}{\big| (e_n)^\Omega + (e_n)^B \big|}
\]
and pick a base of orthonormal vectors $\{(e_1)^I, (e_2)^I, ... , (e_{n-1})^I\}$ which generates the tangent plane $\mathrm{P}_I(w_0)$ at $w_0$ with respect to normal $(e_n)^I$; see Figure \ref{Ftan}.
We shall prove that with $(e_n)^I$ representing the direction of the $n$-axis, this base $\{(e_1)^I, (e_2)^I, ... , (e_n)^I\}$ is indeed the coordinate system in $w_0$ that we are seeking. 
\begin{figure}[htb]
\centering
\includegraphics[width=8.0cm]{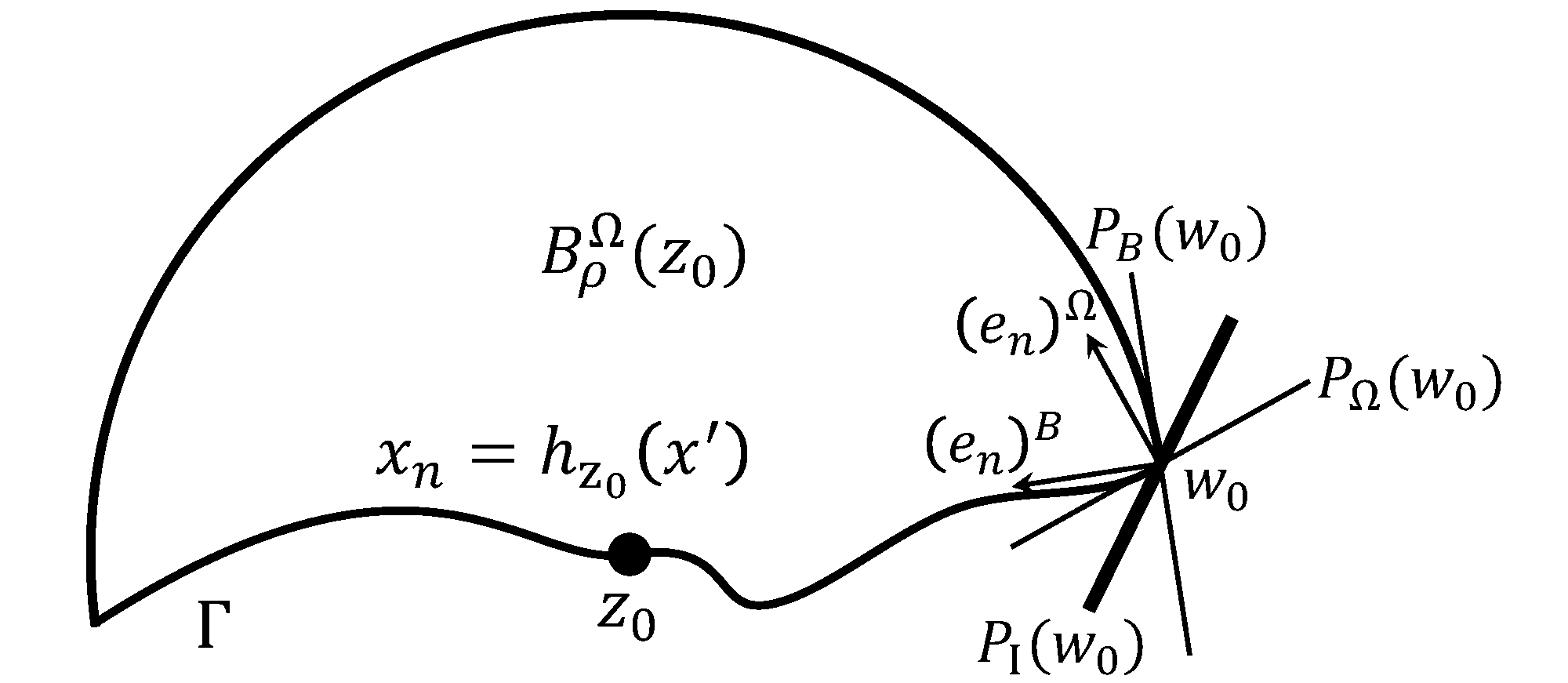}
\caption{$P_I(w_0)$} \label{Ftan}
\end{figure}

In addition, we pick two more bases of orthonormal vectors, say $\{(e_1)^\Omega, (e_2)^\Omega, ... , (e_{n-1})^\Omega\}$ and $\{(e_1)^B, (e_2)^B, ... , (e_{n-1})^B\}$, which generates the tangent plane $\mathrm{P}_\Omega(w_0)$ at $w_0$ with respect to normal $(e_n)^\Omega$ and the tangent plane $\mathrm{P}_B(w_0)$ at $w_0$ with respect to normal $(e_n)^B$ respectively. Within this proof, for a point $u \in \mathbf{R}^n$, we say that $u = (u_1,u_2, ... , u_n)$ is the coordinate of $u$ in the coordinate system in $z_0$, $u^\Omega = (u^\Omega_1, u^\Omega_2, ... , u^\Omega_n)$ is the coordinate of $u$ in the coordinate system in $w_0$ where $(e_n)^\Omega$ represents the direction of the $n$-axis, $u^B=(u^B_1,u^B_2, ... ,u^B_n)$ is the coordinate of $u$ in the coordinate system in $w_0$ where $(e_n)^B$ represents the direction of the $n$-axis and $u^I=(u^I_1,u^I_2, ... ,u^I_n)$ is the coordinate of $u$ in the coordinate system in $w_0$ where $(e_n)^I$ represents the direction of the $n$-axis. 
Since $u$ is nothing but the same point in different coordinate system, we have that
\begin{align} \label{CUD}
u = (u_1,u_2, ... , u_n) = w_0 + \sum_{i=1}^n u^\Omega_i \cdot (e_i)^\Omega = w_0 + \sum_{i=1}^n u^B_i \cdot (e_i)^B = w_0 + \sum_{i=1}^n u^I_i \cdot (e_i)^I.
\end{align}
\begin{proof}[Proof of Lemma \ref{BOL}]
Let $\pi_I : \mathbf{R}^n \to P_I(w_0)$ be the projection of points in $\mathbf{R}^n$ onto the tangent plane $P_I(w_0)$, i.e., for any $u \in \mathbf{R}^n$ we have that either $u - \pi_I(u) = |u-\pi_I(u)| (e_n)^I$ or $u - \pi_I(u) = - |u-\pi_I(u)| (e_n)^I$.
Since $\big( \pi_I(u) \big)^I = (u^{I,'},0)$ where $u^{I,'} := (u^I_1, u^I_2, ... , u^I_{n-1})$, for any $z,y \in \mathbf{R}^n$ in the coordinate system in $z_0$, coordinate equivalence relation (\ref{CUD}) tells us that
\[
|\pi_I(z) - \pi_I(y)| = |z^{I,'} - y^{I,'}| \leq |z-y|,
\]
which means that $\pi_I : \mathbf{R}^n \to P_I(w_0)$ is a Lipschitz continuous map.

We claim that $\pi_I : B^\partial_{\rho^2}(w_0) \setminus \Gamma \to \pi_I\big( B^\partial_{\rho^2}(w_0) \setminus \Gamma \big)$ is injective where $B^\partial_{\rho^2}(w_0) := \{ u \in \partial B^\Omega_\rho(z_0) \bigm| |u-w_0|<\rho^2 \}$.
Let us consider the equation $\big| y + \ell (e_n)^I \big| = \rho$ for $y \in B^\partial_{\rho^2}(w_0) \setminus \Gamma$ and $\ell \in (-\infty,\infty)$. Then, we have that
\[
|y + \ell (e_n)^I|^2 = |y|^2 + 2 \ell \big( y \cdot (e_n)^I \big) + \ell^2 = \rho^2,
\]
which further implies that either $\ell = 0$ or 
\[
\ell = - \frac{2 y \cdot \big( (e_n)^\Omega + (e_n)^B \big)}{\big| (e_n)^\Omega + (e_n)^B \big|} 
\]
as $|y| = \rho$.
The mean value theorem tells us that $|\nabla' h_{z_0} (w_0')| \leq nK |w_0'|$ and $| h_{z_0}(w_0')| \leq nK |w_0'|^2$.
Hence, by making use of the explicit expressions for $(e_n)^\Omega$ and $(e_n)^B$, we can deduce that
\begin{align} \label{EOB}
\big| (e_n)^\Omega + (e_n)^B \big|^2 \leq 2 + \frac{2 \big| \nabla' h_{z_0}(w_0') \big| \cdot |w_0'|}{\rho \big( 1+\big| \nabla' h_{z_0}(w_0') \big|^2 \big)^{1/2}} + \frac{2\big| h_{z_0}(w_0') \big|}{\rho \big( 1+\big| \nabla' h_{z_0}(w_0') \big|^2 \big)^{1/2}} \leq 2 + 4nK \rho.
\end{align}
On the other hand, 
\[
y \cdot \big( (e_n)^\Omega + (e_n)^B \big) = (y-w_0) \cdot \big( (e_n)^\Omega + (e_n)^B \big) + w_0 \cdot (e_n)^\Omega + w_0 \cdot (e_n)^B.
\]
Similar to the derivation of estimate (\ref{EOB}), we can deduce that $|w_0 \cdot (e_n)^\Omega| \leq 2nK \rho^2$. Since $w_0 \cdot (e_n)^B = - \rho$, $\rho < (4nK+4)^{-1}$ and $|y-w_0| < \rho^2$, we then obtain that
\begin{align} \label{EPL}
- y \cdot \big( (e_n)^\Omega + (e_n)^B \big) \geq \rho - (2nK+2)\rho^2 \geq \frac{\rho}{2},
\end{align}
which leads to the fact that $\ell \geq \frac{\rho}{\sqrt{3}}$. That is to say, $\pi_I : B^\partial_{\rho^2}(w_0) \setminus \Gamma \to \pi_I\big( B^\partial_{\rho^2}(w_0) \setminus \Gamma \big)$ is indeed injective since for any $y, y_\ast \in B^\partial_{\rho^2}(w_0) \setminus \Gamma$, we have that $|y - y_\ast| < 2\rho^2 < \frac{\rho}{\sqrt{3}}$.

In the meantime, $\pi_I : B^\partial_{\rho^2}(w_0) \cap B_\rho(z_0) \to \pi_I\big( B^\partial_{\rho^2}(w_0) \cap B_\rho(z_0) \big)$ is also injective. Indeed as otherwise, there exist $z, z_\ast \in B^\partial_{\rho^2}(w_0) \cap B_\rho(z_0)$ and $\ell \in (0,\infty)$ such that $z + \ell (e_n)^I = z_\ast$. If we focus on the $n$-th component of this vector equation, we see that
\[
z_n + \ell (e_n)^I_n = h_{z_0}(z') + \ell (e_n)^I_n = h_{z_0}(z_\ast') = (z_\ast)_n.
\]
Since $\ell = |z_\ast - z| \geq |z_\ast' - z'|$, by the mean value theorem, there exists $t \in (0,1)$ such that
\[
\big| \nabla' h_{z_0}\big( (1-t) z' + t z_\ast' \big) \big| \geq \frac{| h_{z_0}(z_\ast') - h_{z_0}(z') |}{| z_\ast' - z' |} = \frac{\ell (e_n)^I_n}{|z_\ast' - z'|} \geq (e_n)^I_n.
\]
By a similar derivation as for estimate (\ref{EOB}), we can deduce that
\begin{equation} \label{ENIN}
\begin{split}
(e_n)^I_n &\geq \frac{1}{\big| (e_n)^\Omega + (e_n)^B \big|} \cdot \left( \frac{1}{\big( 1+ \big| \nabla' h_{z_0}(w_0') \big|^2 \big)^{\frac{1}{2}}} - \bigg| \frac{| h_{z_0}(w_0') |}{\rho} \bigg| \right) \\
&\geq \frac{1}{\big( 2 + 4nK \rho \big)^{\frac{1}{2}}} \cdot \left( \frac{1}{\big( 1+ n^2 K^2 \rho^2 \big)^{\frac{1}{2}}} - nK \rho \right) > \frac{2}{5}
\end{split}
\end{equation}
for any $\rho<(4nK)^{-1}$. On the other hand, since both $|z_\ast'|$ and $|z'|$ are less than $\rho$ and the open ball $B_\rho(0')$ is convex, the line segment joining $z_\ast'$ and $z'$ lies completely in $B_\rho(0')$, i.e., we have that $\big| (1-t) z' + t z_\ast' \big| < \rho$. Hence, $\big| \nabla' h_{z_0}\big( (1-t)z' + t z_\ast' \big) \big|$ can be estimated by $nK \rho$, which is less than $1/4$ as $\rho<(4nK)^{-1}$. We obtain a contradiction. The mapping $\pi_I : B^\partial_{\rho^2}(w_0) \cap B_\rho(z_0) \to \pi_I\big( B^\partial_{\rho^2}(w_0) \cap B_\rho(z_0) \big)$ is indeed injective. 

Moreover, we could also prove that for any $z \in B^\partial_{\rho^2}(w_0) \cap B_\rho(z_0)$ and $y \in B^\partial_{\rho^2}(w_0) \setminus \Gamma$, we have that $\pi_I(z) \neq \pi_I(y)$. Suppose that there exist $\ell \in (0,\infty)$, $z \in B^\partial_{\rho^2}(w_0) \cap B_\rho(z_0)$ and $y \in B^\partial_{\rho^2}(w_0) \setminus \Gamma$ such that $y + \ell (e_n)^I = z$. Since $\rho < \mathrm{min} \{ \alpha, \beta, R_\ast\}$, the $n$-th component of this vector equation tells us that
\[
h_{z_0}(z') = \ell (e_n)^I_n + y_n \geq \ell (e_n)^I_n + h_{z_0}(y')
\]
as $y \in B^\partial_{\rho^2}(w_0) \setminus \Gamma$ lies above $\Gamma$. Then same argument in the above paragraph would lead us to the same contradiction. Suppose that there exist $\ell_\ast \in (0,\infty)$, $z \in B^\partial_{\rho^2}(w_0) \cap B_\rho(z_0)$ and $y \in B^\partial_{\rho^2}(w_0) \setminus \Gamma$ such that $z + \ell_\ast (e_n)^I = y$. Note that a straight line would intersect a sphere at two points at most. Let us consider the straight line $L_y := \{ y + \ell (e_n)^I \bigm| \ell \in (-\infty,\infty) \}$. The equation $z + \ell_\ast (e_n)^I = y$ means that the straight line $L_y$ actually has intersection with the sphere $\partial B_\rho(z_0)$. Clearly $L_y$ does not lie in the tangent plane at $y$ with respect to the sphere $\partial B_\rho(z_0)$. Indeed, since the normal at $y$ that points inside the ball $B_\rho(z_0)$ is $-\frac{y}{\rho}$ and
\[
\bigg| -\frac{y}{\rho} \cdot (e_n)^I \bigg| = \bigg| \frac{w_0 - y}{\rho} \cdot (e_n)^I + (e_n)^B \cdot (e_n)^I \bigg| \geq \frac{1 - 2nK \rho}{\big( 2 + 4nK \rho \big)^{\frac{1}{2}}} - \rho > \frac{1}{2 \sqrt{3}} - \frac{1}{4} > 0
\]
as $\rho < (4nK + 4)^{-1}$. Thus, the straight line $L_y$ must intersects $\partial B_\rho(z_0)$ at two points. One point is $y$, the other point must be of the form $y - \ell_1 (e_n)^I$ for some $\ell_1 > 0$ as the ray $L_y^- := \{ y - \ell (e_n)^I \bigm| \ell > 0 \}$, which starts from $y$, enters the open ball $B_\rho(z_0)$. This is due to the fact that the point $z = y - \ell_\ast (e_n)^I$ actually lies inside $B_\rho(z_0)$. Hence, the ray $L_y^-$ must cross the sphere $\partial B_\rho(z_0)$ again in order to get out of $B_\rho(z_0)$. On the other hand, since $y \in B^\partial_{\rho^2}(w_0) \setminus \Gamma$ and $L_y$ must intersects $\partial B_\rho(z_0)$ at two points, estimate (\ref{EPL}) and (\ref{EOB}) implies the fact that the ray $L_y^+ := \{ y + \ell (e_n)^I \bigm| \ell > 0 \}$ starting from $y$ would intersect $\partial B_\rho(z_0)$ again for some $\ell_2 > \frac{\rho}{\sqrt{3}}>0$. Now, we arrived at the fact that the straight line $L_y$ intersects the sphere $\partial B_\rho(z_0)$ at three different points; i.e., $y - \ell_1 (e_n)^I$, $y$ and $y + \ell_2 (e_n)^I$ with $\ell_1,\ell_2 > 0$, which is clearly a contradiction.
Therefore, together with the fact that $\pi_I : B^\partial_{\rho^2}(w_0) \setminus \Gamma \to \pi_I\big( B^\partial_{\rho^2}(w_0) \setminus \Gamma \big)$ and $\pi_I : B^\partial_{\rho^2}(w_0) \cap B_\rho(z_0) \to \pi_I\big( B^\partial_{\rho^2}(w_0) \cap B_\rho(z_0) \big)$ are both injective, we conclude that the mapping $\pi_I : B^\partial_{\rho^2}(w_0) \to \pi_I\big( B^\partial_{\rho^2}(w_0) \big)$ is indeed injective.

For $x,z \in \Gamma \cap \overline{B_\rho(z_0)}$, the coordinate equivalence relation (\ref{CUD}) tells us that $|z-x|=|z^I - x^I|$. By the triangle inequality, we see that
\begin{align} \label{TIEC}
|z^I_n - x^I_n| \leq |z^{I,'} - x^{I,'}| + |x-z|.
\end{align}
For $1 \leq j \leq n-1$, note that
\[
z_j-x_j = \sum_{i=1}^{n-1} ( z^I_i - x^I_i ) (e_i)^I_j + ( z^I_n - x^I_n ) (e_n)^I_j
\]
where the notation $(e_i)^I_k$ denotes the $k$-th component of unit vector $(e_i)^I$ for $1 \leq i,k \leq n$. Similar to the derivation of estimate (\ref{EOB}), we can deduce that
\begin{align*}
\big| (e_n)^I_j \big| &\leq \frac{1}{|(e_n)^\Omega + (e_n)^B|} \left( \frac{\big| \partial_j h_{z_0} (w_0') \big|}{\big( 1 + \big| \nabla' h_{z_0} (w_0') \big|^2 \big)^{1/2}} + \frac{\big| (w_0)_j \big|}{\rho} \right) \\
&\leq \frac{1}{(2-4nK \rho)^{\frac{1}{2}}} \bigg( K |w_0'| + \frac{\big| (w_0)_j \big|}{\rho} \bigg).
\end{align*}
Note that by the Cauchy-Schwarz inequality, we have the estimate
\begin{align*}
&\left| \sum_{j=1}^{n-1} \bigg( \sum_{i=1}^{n-1} (z^I_i - x^I_i) (e_i)^I_j \bigg) (z^I_n - x^I_n)(e_n)^I_j \right| \\
&\ \ \leq \left( \sum_{j=1}^{n-1} \bigg( \sum_{i=1}^{n-1} (z^I_i - x^I_i) \cdot (e_i)^I_j \bigg)^2 \right)^{\frac{1}{2}} \big| z^I_n - x^I_n \big| \bigg( \sum_{j=1}^{n-1} \big| (e_n)^I_j \big|^2 \bigg)^{\frac{1}{2}}.
\end{align*}
Hence, by summing up all $j$ from $1$ to $n-1$, we obtain that
\begin{align*}
|z'-x'| &\leq \left( \sum_{j=1}^{n-1} \bigg( \sum_{i=1}^{n-1} (z^I_i - x^I_i) \cdot (e_i)^I_j \bigg)^2 \right)^{\frac{1}{2}} + |z^I_n - x^I_n| \cdot \bigg( \sum_{j=1}^{n-1} \big| (e_n)^I_j \big|^2 \bigg)^{\frac{1}{2}} \\
&\leq n | z^{I,'} - x^{I,'} | + \bigg( \frac{nK^2 \rho^2 + 2nK \rho +1}{2-4nK \rho} \bigg)^{\frac{1}{2}} |z^I_n - x^I_n|.
\end{align*}
On the other hand, by the mean value theorem we see that
\[
|z_n - x_n| = \big| h_{z_0}(z') - h_{z_0}(x') \big| \leq nK \rho |z' - x'|,
\]
which further implies that
\begin{equation} \label{EZMX}
\begin{split}
|z-x| &\leq (nK \rho +1) |z'-x'| \\ 
&\leq (n^2 K \rho + n) |z^{I,'} - x^{I,'}| + (nK \rho +1) \bigg( \frac{nK^2 \rho^2 + 2nK \rho +1}{2-4nK \rho} \bigg)^{\frac{1}{2}} |z^I_n - x^I_n|.
\end{split}
\end{equation}
Since $\rho<(12nK)^{-1}$, then
\[
(nK \rho +1) \bigg( \frac{nK^2 \rho^2 + 2nK \rho +1}{2-4nK \rho} \bigg)^{\frac{1}{2}} < \frac{169}{48 \sqrt{15}} < \frac{169}{185}.
\]
Therefore, by substituting estimate (\ref{EZMX}) back into estimate (\ref{TIEC}) and then considering the absorption principle, we obtain the estimate
\begin{align} \label{LEG}
|z^I_n - x^I_n| \leq 13n |z^{I,'} - x^{I,'}|.
\end{align}

We next consider $x,y \in B_{\rho^2}^\partial (w_0) \setminus \big( \Gamma \cap B_\rho(z_0) \big)$.
Let $g_{w_0}(u') := \rho - \big( \rho^2 - |u'|^2 \big)^{1/2}$ for $u' \in B_{\rho^2}(0')$. Within the local coordinate system in $w_0$ represented by the orthonormal base $\{(e_1)^B, (e_2)^B, ... , (e_n)^B\}$, we have that $B^\partial_{\rho^2}(w_0) \setminus \Gamma \subset g_{w_0} \big( B_{\rho^2}(0') \big)$,
i.e., $B^\partial_{\rho^2}(p_0) \setminus \Gamma$ can be viewed as the graph of $g_{w_0}$ on some connected open subset of $B_{\rho^2}(0')$. 
Obviously both $x$ and $y$ belong to $g_{w_0}\big( B_{\rho^2}(0') \big)$, which means that $x^B_n = g_{w_0}\big( x^{B,'} \big)$ and $y^B_n = g_{w_0}\big( y^{B,'} \big)$.
By coordinate equivalence condition (\ref{CUD}), we see that $|y^I - x^I| = |y^B - x^B|$. The triangle inequality implies that
\begin{align} \label{TPEC}
|y^I_n - x^I_n| \leq |y^B - x^B| + |y^{I,'} - x^{I,'}|.
\end{align}
By the mean value theorem, we have that
\[
\big| y^B_n - x^B_n \big| \leq \big| g_{w_0}\big( y^{B,'} \big) - g_{w_0}\big( x^{B,'} \big) \big| \leq \frac{1}{3} \big| y^{B,'} - x^{B,'} \big|
\]
as $\rho<1/4$. There exist rotation matrices $R_{z_0}^I$ and $R_{z_0}^B$ that satisfies $R_{z_0}^I e_i = (e_i)^I$ and $R_{z_0}^B e_i = (e_i)^B$ for all $1 \leq i \leq n$. The coordinate equivalence condition (\ref{CUD}) also says that
\[
y-x = R_{z_0}^B \cdot (y^B - x^B) = R_{z_0}^I \cdot (y^I - x^I),
\]
which further implies that
\begin{align} \label{EYMX}
y^B - x^B = \big( R_{z_0}^B \big)^{\mathrm{T}} \cdot R_{z_0}^I \cdot (y^I - x^I),
\end{align}
where $\big( R_{z_0}^B \big)^{\mathrm{T}}$ is the transpose of $R_{z_0}^B$. Note that the $n$-th column of matrix $\big( R_{z_0}^B \big)^{\mathrm{T}} \cdot R_{z_0}^I$ is indeed $\big( (e_1)^B \cdot (e_n)^I, \, (e_2)^B \cdot (e_n)^I, ... \, , \, (e_n)^B \cdot (e_n)^I \big)$. As we have already derived it in estimate (\ref{EOB}), 
\[
\big| (e_n)^B \cdot (e_n)^I \big| = \frac{\big| 1 + (e_n)^B \cdot (e_n)^\Omega \big|}{\big| (e_n)^\Omega + (e_n)^B \big|} \geq \frac{1 - 2nK \rho}{\sqrt{2 + 4nK \rho}}.
\]
Since $\big( R_{z_0}^B \big)^{\mathrm{T}} \cdot R_{z_0}^I$ is also an orthogonal matrix, we have that
\[
\sum_{i=1}^{n-1} \big| (e_i)^B \cdot (e_n)^I \big|^2 = 1 - \big| (e_n)^B \cdot (e_n)^I \big|^2 \leq \frac{1 + 8nK \rho - 4n^2 K^2 \rho^2}{2 + 4nK \rho}.
\]
From the coordinate equivalence relation (\ref{EYMX}), we deduce the estimate
\[
\big| y^B_i - x^B_i \big| \leq (n-1) \big| y^{I,'} - x^{I,'} \big| + \big| (e_i)^B \cdot (e_n)^I \big| \cdot \big| y^I_n - x^I_n \big|
\]
which holds for any $1 \leq i \leq n-1$. 
As the Cauchy-Schwarz inequality guarantees that
\[
\sum_{i=1}^{n-1} \big| (e_i)^B \cdot (e_n)^I \big| \leq (n-1)^{\frac{1}{2}} \bigg( \sum_{i=1}^{n-1} \big| (e_i)^B \cdot (e_n)^I \big|^2 \bigg)^{\frac{1}{2}},
\]
by summing up every $i$ from $1$ to $n-1$, we deduce that
\begin{align*}
\big| y^{B,'} - x^{B,'} \big| \leq (n-1)^{\frac{3}{2}} \big| y^{I,'} - x^{I,'} \big| + \bigg( \frac{1 + 8nK\rho - 4n^2 K^2 \rho^2}{2 + 4nK \rho} \bigg)^{\frac{1}{2}} \big| y^I_n - x^I_n \big|.
\end{align*}
Substitute the estimate for $\big| y^{B,'} - x^{B,'} \big|$ and $\big| y^B_n - x^B_n \big|$ back into estimate (\ref{TPEC}), we have that
\[
\big| y^I_n - x^I_n \big| \leq \bigg( \frac{4}{3} n^{\frac{3}{2}} +1 \bigg) \big| y^{I,'} - x^{I,'} \big| + \frac{4}{3} \cdot \bigg( \frac{1 + 8nK\rho - 4n^2 K^2 \rho^2}{2 + 4nK \rho} \bigg)^{\frac{1}{2}} \big| y^I_n - x^I_n \big|.
\]
Since $\rho< (96nK)^{-1}$, then 
\[
\frac{4}{3} \cdot \bigg( \frac{1 + 8nK\rho - 4n^2 K^2 \rho^2}{2 + 4nK \rho} \bigg)^{\frac{1}{2}} < \frac{4 \sqrt{5}}{9} < \frac{995}{1000}.
\]
Therefore, by absorption principle we obtain that
\begin{align} \label{LEB}
\big| y^I_n - x^I_n \big| \leq 600 n^{\frac{3}{2}} \big| y^{I,'} - x^{I,'} \big|.
\end{align}

Since $\pi_I : B^\partial_{\rho^2}(w_0) \to \pi_I\big( B^\partial_{\rho^2}(w_0) \big)$ is continuous and injective, by the theorem of invariance of domain, see e.g. \cite[Corollary 19.9]{Bre}, $\pi_I$ is a homeomorphism between $B^\partial_{\rho^2}(w_0)$ and $\pi_I\big( B^\partial_{\rho^2}(w_0) \big)$ and it is an open map. Thus, $\pi_I\big( B^\partial_{\rho^2}(w_0) \big)$ is open as $B^\partial_{\rho^2}(w_0)$ is open. There exists $\varepsilon_0 > 0$ such that $B_{\varepsilon_0}(0') \subset \pi_I\big( B^\partial_{\rho^2}(w_0) \big)$. We let $\alpha_0 = \mathrm{min} \, \{ \varepsilon_0/2, \rho^2/(600 n^{\frac{3}{2}}) \}$. The mapping $\pi_I^{-1} : B_{\alpha_0}(0') \to \pi_I^{-1}\big( B_{\alpha_0}(0') \big)$ is continuous and bijective.  
For $z \in \pi_I^{-1}\big( B_{\alpha_0}(0') \big) \cap B_\rho(z_0)$ and $y \in \pi_I^{-1}\big( B_{\alpha_0}(0') \big) \setminus \Gamma$, there exists $x \in \Gamma \cap \partial B_\rho(z_0) \cap B^\partial_{\rho^2}(w_0)$ such that $\pi_I(x) = \big( x^{I,'}, 0 \big)$ lies exactly on the line segment joining $\pi_I(z) = \big( z^{I,'}, 0 \big)$ and $\pi_I(y) = \big( y^{I,'}, 0 \big)$ in the hyperplane $P_I(w_0)$. Indeed, as the line segment joining $\big( z^{I,'}, 0 \big)$ and $\big( y^{I,'}, 0 \big)$ in $P_I(w_0)$ are of the form $\big( z^{I,'}, 0 \big) + t \cdot \big( y^{I,'} - z^{I,'}, 0 \big)$ where $0 \leq t \leq 1$. Let $0< t_\ast < 1$ be the smallest $t$ such that 
\[
\pi_I^{-1}\big( \big( z^{I,'}, 0 \big) + t \cdot \big( y^{I,'} - z^{I,'}, 0 \big) \big) \notin \pi_I^{-1}\big( B_{\alpha_0}(0') \big) \cap B_\rho(z_0)
\]
for the first time. Let $x = \pi_I^{-1}\big( \big( z^{I,'}, 0 \big) + t_\ast \cdot \big( y^{I,'} - z^{I,'}, 0 \big) \big)$. We must have that $x \in \Gamma \cap \partial B_\rho(z_0)$ since $\pi_I^{-1}\big( B_{\alpha_0}(0') \big) \setminus \Gamma$ is open. Therefore, by estimate (\ref{LEG}) and (\ref{LEB}) we deduce that
\begin{align*}
\big| z^I_n - y^I_n \big| &\leq \big| z^I_n - x^I_n \big| + \big| y^I_n - x^I_n \big| \leq 600 n^{\frac{3}{2}} \big( \big| z^{I,'} - x^{I,'} \big| + \big| y^{I,'} - x^{I,'} \big| \big) = 600 n^{\frac{3}{2}} \big| z^{I,'} - y^{I,'} \big|.
\end{align*}
Take $\beta_0 = \alpha_0$ and $\gamma_0(u') := \big( \pi_I^{-1}(u') \big)^I_n$ for $u' \in B_{\alpha_0}(0')$, then we are done.
\end{proof}

Next, let us recall the concept for a bounded domain $D$ to be star-shaped (or star-like). We say that $D$ is star-shaped (or star-like) with respect to a point $\overline{x} \in D$ if every ray starting from $\overline{x}$ intersects $\partial D$ at one and only one point, and that $D$ is star-shaped (or star-like) with respect to an open ball $B \subset D$ if $D$ is star-shaped (or star-like) with respect to every point $\overline{x} \in B$; see e.g. \cite[Section II.1.4]{Gal}. 

\begin{lemma} \label{BOS}
Let $z_0 \in \Gamma$ and $0<\rho< \mathrm{min} \{ (32nK)^{-1}, \alpha, \beta, R_\ast\}$. Then $B^\Omega_\rho(z_0)$ is star-like with respect to $B_{\frac{\rho}{4}}(x_0) \subset B^\Omega_\rho(z_0)$ where $x_0$ is the point whose projection on $\Gamma$ is $z_0$ and $d(x_0) = \frac{\rho}{2}$.
\end{lemma}
\begin{proof}
For any $x \in B^\Omega_\rho(z_0)$ and $e_v \in \mathbf{S}^n$ where $\mathbf{S}^n$ denotes the $n$ dimensional unit sphere centered at $0$, the ray $L_x^+ := \{ x + \ell e_v \bigm| \ell>0 \}$ would only intersects $\partial B_\rho(z_0)$ once since otherwise, as we have discussed in the proof of Lemma \ref{BOL}, the straight line $L_x := \{ x + \ell e_v \bigm| \ell \in \mathbf{R} \}$ would have intersected $\partial B_\rho(z_0)$ at more than two different points, which is absurd. Hence, it is sufficient to prove that for any $x \in B_{\frac{\rho}{4}}(x_0)$, the ray $L_x$ not only cannot have two intersections on $B_\rho(z_0) \cap \Gamma$, but also cannot have one intersection on $\partial B^\Omega_\rho(z_0) \setminus \Gamma$ and another intersection on $B_\rho(z_0) \cap \Gamma$. We shall prove these two claims by assuming the contrary.

Note that the coordinate of $x_0$ in the local coordinate system in $z_0$ is indeed $(0', \frac{\rho}{2})$. Hence within this proof, for any $u \in B_\rho(z_0)$, we may assume that $u = (u_1, ... , u_n)$ is the coordinate of point $u$ in the local coordinate system in $z_0$. For $x \in B_{\frac{\rho}{4}}(x_0)$, the triangle inequality implies that
\[
x_n \geq \frac{\rho}{2} - \bigg| x_n - \frac{\rho}{2} \bigg| \geq \frac{\rho}{2} - |x - x_0| > \frac{\rho}{4}.
\]
Hence, for $x \in B_{\frac{\rho}{4}}(x_0)$ with $\ell>0$, $e_v \in \mathbf{S}^n$ and $z' \in B_\rho(0')$ such that $x + \ell e_v = \big( z', h_{z_0}(z') \big)$, we then must have
\[
- \ell (e_v)_n = x_n - h_{z_0}(z') > \frac{\rho}{4} - nK \rho^2 > \frac{\rho}{8} > 0,
\]
i.e., in this case the $n$-th component of $e_v$ must be negative. Since $x + \ell e_v = \big( z', h_{z_0}(z') \big) \in B_\rho(z_0)$, we must also have that $\ell < 2\rho$. Thus, we can deduce that $-(e_v)_n > \frac{1}{16}$.

Suppose that there exist $x \in B_{\frac{\rho}{4}}(x_0)$, $e_v \in \mathbf{S}^n$ and $0< \ell_1 < \ell_2$ such that the ray $L_x^+$ intersects $B_\rho(z_0) \cap \Gamma$ at points $x + \ell_1 e_v$ and $x + \ell_2 e_v$. That means that there exist $z_1', z_2' \in B_\rho(0')$ such that
\[
x + \ell_1 e_v = \big( z_1', h_{z_0}(z_1') \big), \quad x + \ell_2 e_v = \big( z_2', h_{z_0}(z_2') \big).
\]
By subtracting $x + \ell_1 e_v$ from $x + \ell_2 e_v$, we can easily deduce that
\[
\frac{\big| h_{z_0}(z_1') - h_{z_0}(z_2') \big|}{|z_1' - z_2'|} = - \frac{(\ell_2 - \ell_1) (e_v)_n}{|z_1' - z_2'|} > \frac{1}{16}.
\]
Hence, similar as in the proof of Lemma \ref{BOL}, by the mean value theorem we may deduce that
\[
\frac{1}{32} > nK \rho \geq \| \nabla' h_{z_0} \|_{L^\infty\big(B_\rho(0') \big)} > \frac{1}{16},
\]
which is a contradiction. Therefore, we conclude that for any $x \in B_{\frac{\rho}{4}}(x_0)$ and $e_v \in \mathbf{S}^n$, the ray $L_x^+$ only intersects $B_\rho(z_0) \cap \Gamma$ once.

Suppose that there exist $x \in B_{\frac{\rho}{4}}(x_0)$, $e_v \in \mathbf{S}^n$ and $0<\ell_1<\ell_2$ such that the ray $L_x^+$ intersects $B_\rho(z_0) \cap \Gamma$ at $x+\ell_1 e_v$ and then $\partial B^\Omega_\rho(z_0) \setminus \Gamma$ at $x + \ell_2 e_v$. Since the $n$-th component of $x + \ell_1 e_v$ is of the form $\big( z', h_{z_0}(z') \big)$ for some $z' \in B_\rho(0')$. Then we must have that
\[
h_{z_0}(z') + (\ell_2 - \ell_1) (e_v)_n = x_n + \ell_2 (e_v)_n > h_{z_0}\big( x' + \ell_2 (e_v)' \big),
\]
with $(e_v)' := \big( (e_v)_1, ... , (e_v)_{n-1} \big)$ since $x + \ell_2 e_v \in \partial B_\rho(z_0) \setminus \Gamma$ lies above $\Gamma$. Similar as above, by mean value theorem we deduce that
\[
\frac{1}{32} > \| \nabla' h_{z_0} \|_{L^\infty\big( B_\rho(0') \big)} \geq \frac{\big| h_{z_0}(z') - h_{z_0}\big( x' + \ell_2 (e_v)' \big) \big|}{\big|z' - x' - \ell_2 (e_v)' \big|} > \frac{\ell_1 - \ell_2}{\big| z' - x' - \ell_2 (e_v)' \big|} (e_v)_n \geq -(e_v)_n,
\]
which is a contradiction as $-(e_v)_n > \frac{1}{16}$. Therefore, we see that for any $x \in B_{\frac{\rho}{4}}(x_0)$ and $e_v \in \mathbf{S}^n$, if the ray $L_x^+$ intersects $B_\rho(z_0) \cap \Gamma$ at $x + \ell_\ast e_v$ for some $\ell_\ast>0$, then $L_x^+$ does not intersect $\partial B^\Omega_\rho(z_0)$ for any $\ell > \ell_\ast$.

Finally, for $x \in B_\rho(z_0)$, we suppose that there exist $e_v \in \mathbf{S}^n$ and $0<\ell_1<\ell_2$ such that the ray $L_x^+$ intersects $\partial B^\Omega_\rho(z_0) \setminus \Gamma$ at $x + \ell_1 e_v$ and then $B_\rho(z_0) \cap \Gamma$ at $x+\ell_2 e_v$. Since $x + \ell_2 e_v \in B_\rho(z_0)$, the ray $L_{x + \ell_2 e_v}^+ := \{ x + \ell e_v \bigm| \ell > \ell_2 \}$ must intersect again at $\partial B_\rho(z_0)$ for some $\ell_3>\ell_2$. On the other hand, since $x \in B_\rho(z_0)$, the ray $L_x^- := \{ x - \ell e_v \bigm| \ell>0 \}$ must intersects $\partial B_\rho(z_0)$ for some $\ell_0>0$. Hence, the straight line $L_x$ intersects $\partial B_\rho(z_0)$ at three different points, i.e., $x - \ell_0 e_v$, $x + \ell_2 e_v$ and $x + \ell_3 e_v$, which is a contradiction. Therefore, for any $x \in B_\rho(z_0)$ and $e_v \in \mathbf{S}^n$, if $L_x^+$ intersects $\partial B^\Omega_\rho(z_0) \setminus \Gamma$ at $x+\ell_\ast e_v$ for some $\ell_\ast>0$, then $L_x^+$ does not intersect $\partial B^\Omega_\rho(z_0)$ for any $\ell>\ell_\ast$.
\end{proof}

\subsection{Estimate for the reach $R_\ast$} \label{sub:ER}
We shall derive a lower bound estimate for the reach of $\Gamma$ in terms of the regularity of $\Gamma$. 
\begin{proposition} \label{DR}
Let $\Omega \subset \mathbf{R}^n$ be a uniformly $C^2$ domain of type $(\alpha,\beta,K)$ and $R_\ast$ be the reach of $\Gamma = \partial \Omega$. There exists a constant $C=C(\alpha,\beta,K)>0$ such that the lower bound estimate $R_\ast \geq C>0$ holds.
\end{proposition}
\begin{proof}
Let $z_0 \in \Gamma$. There exists $\rho_0>0$ such that $B_{2 \rho_0}(z_0) \subset U_{\alpha,\beta,h_{z_0}}(z_0)$ as $U_{\alpha,\beta,h_{z_0}}(z_0)$ is open. Let $x_0 = z_0 - \rho \mathbf{n}(z_0)$ where $\rho \in (0, \rho_0)$ and $\mathbf{n}(z_0)$ denotes the outward unit normal at $z_0$ with respect to $\Gamma$. In the local coordinate system in $z_0$, the coordinate of $x_0$ is indeed $(0', \rho)$. Suppose that there exists $\big( z', h_{z_0}(z') \big) \in U_{\alpha,\beta,h_{z_0}}(z_0) \cap \Gamma$ such that $\big| \big( z', h_{z_0}(z') - \rho \big) \big| \leq \rho$ and $z' \neq 0'$. By the mean value theorem, we have that $\big| h_{z_0}(z') \big| \leq nK |z'|^2$. Hence, we deduce that
\[
|z'|^2 - n^2 K^2 |z'|^4 \leq |z'|^2 + h_{z_0}(z')^2 \leq 2\rho h_{z_0}(z') \leq 2nK \rho |z'|^2,
\]
which further implies that $1 - n^2 K^2 |z'|^2 \leq 2nK \rho$.
By the triangle inequality, we see that $\big| \big( z', h_{z_0}(z') \big) \big| \leq 2 \rho$. Hence, it holds that $1 \leq 2nK \rho + 4 n^2 K^2 \rho^2$, i.e., if $\rho< \mathrm{min} \big\{ \frac{1}{8nK}, \rho_0 \big\}$, for any $z \in U_{\alpha,\beta,h_{z_0}}(z_0) \cap \Gamma$ such that $z \neq z_0$, we must have that $|z - x_0| > \rho$.

Suppose that there exists $w \in U_{\alpha,\beta,h_{z_0}}(z_0) \cap \Gamma$ such that the coordinate of $w$ in the local coordinate system in $z_0$ is of the form $\big( w', h_{z_0}(w') - \nu \big)$ for some $\nu \in (0, \beta)$ and $\big| \big( w', h_{z_0}(w') - \nu - \rho \big) \big| \leq \rho$. Then, we have that
\begin{align} \label{NRHW}
|w'|^2 + h_{z_0}(w')^2 + \nu^2 + 2 \nu \rho \leq 2 (\nu + \rho) h_{z_0}(w').
\end{align}
Since the closed ball $\overline{B_\rho\big( (0',\rho) \big)}$ lies completely in the upper half space $\mathbf{R}_+^n := \{ (x',x_n) \in \mathbf{R}^n \bigm| x_n >0 \}$ except the boundary point $0$, we must have that $w' \neq 0'$ and $h_{z_0}(w') > \nu > 0$. Hence, estimate (\ref{NRHW}) implies that
\begin{align} \label{EWH}
|w'|^2 \leq h_{z_0}(w')^2 + 2 \rho h_{z_0}(w').
\end{align}
Since $\big| h_{z_0}(w') \big| \leq nK |w'|^2$, estimate (\ref{EWH}) implies that $1 \leq n^2 K^2 \rho^2 + 2 n K \rho$, which is a contradiction for $\rho< \mathrm{min} \big\{ \frac{1}{8nK}, \rho_0 \big\}$. Therefore, we show that for any $\rho< \mathrm{min} \big\{ \frac{1}{8nK}, \rho_0 \big\}$, $z_0$ is the unique point in $\Gamma$ such that $|x_0 - z_0 | = d_\Gamma(x_0) = \inf_{z \in \Gamma} \, |x_0 - z|$.

Next, we shall provide a uniform lower bound for $R_\ast$. For $z_0 \in \Gamma$, we define that
\[
\mathrm{reach}\big( \Gamma, z_0 \big) := \underset{\ell>0}{\sup} \, \{ \ell \bigm| z_0 - \ell \mathbf{n}(z_0) \; \, \text{has a unique projection on} \; \, \Gamma\}.
\]
We show in the above two paragraphs that for any $z_0 \in \Gamma$, there exists $\rho_0>0$ such that 
\[
\mathrm{reach}\big( \Gamma, z_0 \big) \geq \mathrm{min} \bigg\{ \frac{1}{8nK}, \rho_0 \bigg\} > 0.
\]
Hence, for any $x_0 = z - \ell \mathbf{n}(z_0)$ with $0<\ell<\mathrm{reach}\big( \Gamma, z_0 \big)$ and $y \in \Gamma$, there exists a lower bound for $\mathrm{reach}\big( \Gamma, z_0 \big)$, i.e.,
\begin{align} \label{RGz}
\mathrm{reach}\big( \Gamma, z_0 \big) \geq - \frac{|z_0-y|^2 |x_0-z_0|}{2 (x_0-z_0) \cdot (z_0-y)},
\end{align}
see e.g. \cite[Theorem 4.8.(7)]{Fed}. 
Suppose that there exists $|y'|<\alpha$ such that $h_{z_0}(y')>0$. Since $h_{z_0}(y') \leq n K |y'|^2$, then estimate (\ref{RGz}) would imply that
\[
\mathrm{reach}\big( \Gamma, z_0 \big) \geq \frac{1}{2} \left( \frac{|y'|^2}{h_{z_0}(y')} + h_{z_0}(y') \right) \geq \frac{1}{2nK}.
\]
Suppose that $h_{z_0}(y') \leq 0$ for any $|y'|< \alpha$, then we claim that $\mathrm{reach}\big( \Gamma, z_0 \big) \geq \frac{\beta}{4}$. Suppose that there exists $z_\ast \in \Gamma$ such that $z_\ast \neq z_0$ and $\big| (0',\ell) - z_\ast \big| \leq \ell$ for some $0<\ell<\frac{\beta}{4}$. 
Since $|z_\ast - z_0| \leq 2 \ell$ and the closed ball $\overline{B_\ell\big( (0',\ell) \big)}$ lies completely in $\mathbf{R}_+^n$ except $z_0$, we deduce that
\[
0<(z_\ast)_n< \frac{\beta}{2}, \quad \frac{\beta}{2} - h_{z_0}(z_\ast') \leq \frac{\beta}{2} + n K |z_\ast'|^2 \leq \frac{\beta}{2} + \frac{nK \beta^2}{4}.
\] 
Since we are assuming that $\beta<\frac{2}{nK}$, we have that $\frac{\beta}{2} \leq h_{z_0}(z_\ast') + \frac{\beta}{2}$, which implies that
\[
h_{z_0}(z_\ast') \leq 0 < (z_\ast)_n < \frac{\beta}{2} \leq h_{z_0}(z_\ast') + \frac{\beta}{2}.
\]
Hence, we must have that $z_\ast \in \Omega$ as 
\[
U_{\alpha,\beta,h_{z_0}}(z_0) \cap \Omega = \{ (z',z_n) \bigm| |z'|<\alpha, h_{z_0}(z') < z_n < h_{z_0}(z') + \beta \},
\]
which is a contradiction. For any $z_0 \in \Gamma$, we conclude that $\mathrm{reach}\big( \Gamma, z_0 \big) \geq \frac{\beta}{4}$ as $\beta< \frac{2}{nK}$. Note that $R_\ast \geq \inf_{z_0 \in \Gamma} \, \mathrm{reach}\big( \Gamma, z_0 \big)$, we are done.
\end{proof}

Let us recall the concept for a perturbed $C^k$ $(k \geq 2)$ half space $\mathbf{R}_h^n$ to have small perturbation.
We say that $\mathbf{R}_h^n$ is a perturbed $C^k$ half space if
\[
\mathbf{R}_h^n = \{ (x', x_n) \in \mathbf{R}^n \bigm| x_n > h(x') \}
\]
with $h \in C_{\mathrm{c}}^k(\mathbf{R}^{n-1})$. Let $R_h>0$ be such that $\operatorname{supp} h \subseteq \overline{B_{R_h}(0')}$. 
We say that the perturbed $C^k$ $(k \geq 2)$ half space $\mathbf{R}^n_h$ has small perturbation if 
\begin{align} \label{SC}
R_h^{\frac{2n-1}{2n}} < \frac{1}{2}, \quad
C_s(h)^{\frac{3n}{2} + 8} C_1(h) \left( C_{\ast,1}(h) + C_{\ast,2}(h) + R_h^{\frac{n}{2}} \right) < \frac{1}{2C^\ast(n)}
\end{align}
where $C^\ast(n)$ is a specific constant depending only on the space dimension $n$,
\begin{align*}
	&C_s(h) := 1 + \lVert h \rVert_{C^1(\mathbf{R}^{n-1})}, \quad 
	C_1(h) := 1 + R_h \left\lVert \nabla'^2 h \right\rVert_{L^\infty(\mathbf{R}^{n-1})}, \\
	&C_{\ast,1}(h) := C_1(h)^3 \left(1 + R_h^{\frac14}\right) \left( R_h^{\frac12} \lVert \nabla'^2 h \rVert_{L^\infty(\mathbf{R}^{n-1})} + R_h^{\frac52} \lVert \nabla'^2 h \rVert_{L^\infty(\mathbf{R}^{n-1})}^3 \right), \\
	&C_{\ast,2}(h) := \left( R_h + R_h^{\frac{1}{2h}} \right) \left\lVert \nabla'^2 h \right\rVert_{L^\infty(\mathbf{R}^{n-1})} + \left(R_h^{n-1} + 1\right) \lVert h \rVert_{C^1(\mathbf{R}^{n-1})}.
\end{align*}
In addition, we say a perturbed $C^k$ ($k\geq2$) half space $\mathbf{R}^n_h$ is of type $(K)$ if
\[
\sup_{x' \in \mathbf{R}^{n-1}} \left\lvert \nabla'^2 h (x') \right\rvert < K;
\]
see \cite{GigaGuP}.

\begin{remark} \label{PHS:abk}
If we want to say that a perturbed $C^k$ $(k \in \mathbf{N})$ half space $\mathbf{R}_h^n$ is of type $(\alpha,\beta,K)$ in the sense as other general domains, then in this case constants $\alpha$ and $\beta$ can be taken to be arbitrarily large.
\end{remark}

Let $\Omega \subset \mathbf{R}^n$ be a uniformly $C^2$ domain of type $(\alpha, \beta, K)$ and $R_\ast$ be the reach of boundary $\Gamma = \partial \Omega$.
For $z_0 \in \Gamma$ and $0<\rho<\mathrm{min} \, \big\{ \frac{\beta}{4}, \frac{R_\ast}{2} \big\}$, we have that $B^\Omega_{2 \rho}(z_0) \subset U_{\alpha,\beta,h_{z_0}}(z_0)$. Hence, within the coordinate system we choose in $z_0$, we see that $B_{2 \rho}(z_0) \cap \Gamma \subset h_{z_0}\big( B_{2 \rho}(0') \big)$. 
By considering the smooth function $\zeta$ defined by expression (\ref{CFZ}), we can construct $\theta \in C_{\mathrm{c}}^\infty(\mathbf{R}^{n-1})$ such that $\theta = 1$ in $B_1(0')$ and $\theta = 0$ in $B_2(0')^{\mathrm{c}}$. We then let $\theta_\rho(x') := \theta \big( \frac{x'}{\rho} \big)$ for $x' \in \mathbf{R}^{n-1}$ and $h_{z_0}^\ast := \theta_\rho \cdot h_{z_0}$. A direct calculation implies that
\begin{align*}
&\| \nabla' h_{z_0}^\ast \|_{L^\infty(\mathbf{R}^{n-1})} \leq C(n) \big( 1 + \| \theta \|_{C^1(\mathbf{R}^{n-1})} \big) \rho K, \\
&\| \nabla'^2 h_{z_0}^\ast \|_{L^\infty(\mathbf{R}^{n-1})} \leq C(n) \big( 1 + \| \theta \|_{C^2(\mathbf{R}^{n-1})} \big) K.
\end{align*}
Therefore, if $\rho$ is taken to be sufficiently small, then $h_{z_0}^\ast$ satisfies the smallness condition (\ref{SC}). As a result, we have the following lemma.

\begin{lemma} \label{EBP}
Let $\Omega \subset \mathbf{R}^n$ be a uniformly $C^2$ domain of type $(\alpha, \beta, K)$ and $R_\ast$ be the reach of $\Gamma = \partial \Omega$.
Let
\[
0<\rho< \mathrm{min} \bigg\{ \frac{\beta}{4}, \frac{R_\ast}{2}, c_0(n,K) \bigg\}
\]
where $c_0(n,K)$ is some specific constant depending only on $n$ and $K$.
Then $B^\Omega_\rho(z_0)$ can be viewed as being contained in the perturbed $C^3$ half space $\mathbf{R}_{h_{z_0}^\ast}^n$ of type $\big( \lambda(K,\rho) \big)$ with
\[
\lambda(K,\rho) := C(n) \big( \| \theta \|_{C^3(\mathbf{R}^{n-1})}+1 \big) \frac{K}{\rho}.
\]
Moreover, $\mathbf{R}_{h_{z_0}^\ast}^n$ has small perturbation and the reach of the boundary $\partial \mathbf{R}_{h_{z_0}^\ast}^n$ is greater than or equal to $\frac{1}{2n \lambda(K,\rho)}$.
\end{lemma}
\begin{proof}
It is obvious that $B^\Omega_\rho(z_0) \subset B^\Omega_{2 \rho}(z_0) \subset \mathbf{R}_{h_{z_0}^\ast}^n$. The estimate of $\lambda(K,\rho)$ can be easily deduced since we have that $\| \nabla'^3 \theta_\rho \|_{L^\infty(\mathbf{R}^{n-1})} \leq \| \nabla'^3 \theta \|_{L^\infty(\mathbf{R}^{n-1})} / \rho^3$ and $\big| h_{z_0}(x') \big| \leq nK |x'|^2 \leq 4nK \rho^2$ for $|x'|<2\rho$. 
There exists a constant $c_0(n,K)>0$, depending only on dimension $n$ and $K$, such that $\rho < c_0(n,K)$ implies that $h_{z_0}^\ast$ satisfies the smallness condition (\ref{SC}), i.e., the perturbed $C^3$ half space $\mathbf{R}_{h_{z_0}^\ast}^n$ would have small perturbation. 
Finally by Proposition (\ref{DR}) and Remark \ref{PHS:abk}, we see that the reach of $\partial \mathbf{R}_{h_{z_0}^\ast}^n$ must be greater than or equal to $\frac{1}{2n \lambda(K,\rho)}$.
\end{proof}

\subsection{Bogovski$\breve{\i}$ operator, Morrey's $\&$ Poincar$\acute{\text{e}}$ inequalities} \label{sub:BMP}
Let us recall the existence and boundedness of the Bogovski$\breve{\i}$ operator. Let $D \subset \mathbf{R}^n$ be a bounded domain and $L_0^q(D) := \{ g \in L^q(D) \bigm| \int_D g \, dx = 0 \}$ with $1<q<\infty$. We consider the divergence problem
\[
\operatorname{div} u = g, \quad u \bigm|_{\partial D} = 0
\]
for $g \in L_0^q(D)$. If there exist $x_0 \in D$ and $R>0$ such that $D$ is star-like with respect to $B_R(x_0)$ and $\overline{B_R(x_0)} \subset D$, then there exists a bounded linear operator $B_q: L_0^q(D) \to W_0^{1,q}(D)^n$ satisfying $\operatorname{div} B_q(g) = g$ and
\begin{align} \label{BGOE}
\| B_q(g) \|_{W^{1,q}(D)} \leq C(n,q) \left( \frac{\delta(D)}{R} \right)^n \left( 1+ \frac{\delta(D)}{R} \right) \| g \|_{L^q(D)}
\end{align}
where $\delta(D)$ denotes the diameter of $D$, i.e., $\delta(D) := \underset{x,y \in \overline{D}}{\sup} \, |x-y|$; see e.g. \cite[Lemma III.3.1]{Gal}.

Let $z_0 \in \Gamma$, $n<q \leq \infty$ and $\rho \in (0,R_\ast)$ be sufficiently small. Lemma \ref{BOL} guarantees that $B^\Omega_\rho(z_0)$ is indeed a bounded Lipschitz domain.
By Jones \cite[Theorem 1]{PJS}, we see that there exists a bounded linear extension operator $\Lambda_1^q : W^{1,q}\big( B^\Omega_\rho(z_0) \big) \to W^{1,q}(\mathbf{R}^n)$ which satisfies $\Lambda_1^q(g) \bigm|_{B^\Omega_\rho(z_0)} = g$ and
\[
\| \Lambda_1^q(g) \|_{W^{1,q}(\mathbf{R}^n)} \leq C(L,q,n) \| g \|_{W^{1,q} \big( B^\Omega_\rho(z_0) \big)}
\]
for any $g \in W^{1,q}\big( B^\Omega_\rho(z_0) \big)$, where $L$ is the Lipschitz constant for $B^\Omega_\rho(z_0)$. Let $\varphi_0 \in C_{\mathrm{c}}^\infty(\mathbf{R}^n)$ be such that $\varphi_0 = 1$ in $B_\rho(z_0)$ and $\varphi_0 = 0$ in $B_{2 \rho}(z_0)^{\mathrm{c}}$. A direct calculation implies that
\[
\| \varphi_0 \cdot \Lambda_1^q(g) \|_{W^{1,q}(\mathbf{R}^n)} \leq \frac{C(L,q,n)}{\rho} \| g \|_{W^{1,q} \big( B^\Omega_\rho(z_0) \big)}.
\]
By Morrey's inequality in $\mathbf{R}^n$, we have that
\[
\| \varphi_0 \cdot \Lambda_1^q(g) \|_{C^{0,\gamma}(\mathbf{R}^n)} \leq C(n,q) \| \varphi_0 \cdot \Lambda_1^q(g) \|_{W^{1,q}(\mathbf{R}^n)}
\]
with $\gamma = 1 - \frac{n}{q}$. As $\varphi_0 = 1$ in $B_\rho(z_0)$ and $\Lambda_1^q(g) \bigm|_{B^\Omega_\rho(z_0)} = g$ in $B^\Omega_\rho(z_0)$, we obtain Morrey's inequality in $B^\Omega_\rho(z_0)$, i.e.,
\begin{align} \label{MOR}
\| g \|_{C^{0,\gamma}\big( B^\Omega_\rho(z_0) \big)} \leq \frac{C(L,q,n)}{\rho} \| g \|_{W^{1,q}\big( B^\Omega_\rho(z_0) \big)}
\end{align}
where $\gamma = 1 - \frac{n}{q}$.

Let $z_0 \in \Gamma$ and $n \leq q < \infty$. We finally need the Poincar$\acute{\mathrm{e}}$ inequality in $B^\Omega_\rho(z_0)$ with detailed dependency of the Poincar$\acute{\mathrm{e}}$ constant. 
Let us recall that for $p \geq 1$, a domain $D$ is called an $L^p$-averaging domain if there exists a constant $a_D>0$ such that the estimate
\[
\left( \frac{1}{|D|} \int_D \big| u - u_D \big|^p \, dx \right)^{\frac{1}{p}} \leq a_D [u]_{BMO^\infty(D)}
\]
holds for any $u \in BMO^\infty(D)$; see e.g. \cite{Stap}.
On the other hand, a domain $D$ is called a John domain if there exist $0<\delta \leq 1$ and $x_0 \in D$ such that for any $x \in D$, there is a rectifiable path $\gamma : [0, 1] \to D$ which satisfies $\gamma(0)=x$, $\gamma(1)=x_0$ and for all $t \in [0,1]$,
\[
d\big( \gamma(t), \partial D \big) := \inf_{z \in \partial D} \big| \gamma(t) - z \big| \geq \delta \big| \gamma(t) - \gamma(s) \big|, \quad s \in [0,t];
\]
see e.g. \cite[Lemma 2.7]{MaSa}.
A bounded Lipschitz domain $D$ is a John domain with $\delta$ depending only on the Lipschitz regularity of $\partial D$; see e.g. \cite{Bom}, \cite{Stap}.
A John domain $D$ is indeed an $L^p$-averaging domain with constant $a_D$ depending on $\delta, p$ only; see \cite[Theorem 3.14]{Stap} and \cite[Lemma 2.1]{Bom}.
Hence, in the case where $D$ is a bounded Lipschitz domain, it is guaranteed by \cite[Theorem 3.4]{Stap} that for any $g \in W^{1,q}(D)$, the Poincar$\acute{\mathrm{e}}$ inequality
\begin{align} \label{POI}
\| g - g_D \|_{L^q(D)} \leq C(L_D, q) |D|^{\frac{1}{n}} \| \nabla g \|_{L^q(D)}
\end{align}
holds where 
\[
g_{D} := \frac{1}{|D|} \int_{D} g(y) \, dy
\]
denotes the average value of $g$ in $D$, $|D|$ denotes the Lebesgue measure of domain $D$ and $L_D$ represents the Lipschitz constant which characterizes the Lipschitz regularity of $\partial D$.

The Poincar$\acute{\text{e}}$ inequality for the case of an open ball is well-known. Let $1 \leq q \leq \infty$ and $B \subset \mathbf{R}^n$. For any $g \in W^{1,q}(B)$, we have that
\[
\| g - g_B \|_{L^q(B)} \leq C(n,q) r(B) \| \nabla g \|_{L^q(B)}, \quad g_B = \frac{1}{|B|} \int_B g(y) \, dy,
\]
where $r(B)$ denotes the radius of $B$; see e.g. \cite[\S 5.8, Th. 2]{Evans}.
As a result, for any open domain $D \subseteq \mathbf{R}^n$ and $g \in W^{1,n}(D)$, we have the estimate
\begin{align} \label{EWBMO}
[g]_{BMO^\infty(D)} \leq \sup_{B \subseteq D} \frac{1}{|B|^{\frac{1}{n}}} \| g - g_B \|_{L^n(B)} \leq C(n) \| \nabla g \|_{L^n(D)}.
\end{align}
Indeed, for any open ball $B \subseteq D$, by H$\ddot{\text{o}}$lder's inequality we see that
\[
\| g - g_B \|_{L^1(B)} \leq C(n) r(B)^{n-1} \| g - g_B \|_{L^n(B)}
\]
where $r(B)$ denotes the radius of $B$. Then the Poincar$\acute{\text{e}}$ inequality for an open ball implies that
\[
\frac{1}{|B|} \| g - g_B \|_{L^1(B)} \leq \frac{C(n)}{r(B)} \| g - g_B \|_{L^n(B)} \leq C(n) \| \nabla g \|_{L^n(B)}
\]
for any open ball $B \subseteq D$.
Therefore, estimate (\ref{EWBMO}) follows and we conclude that for any open domain $D \subseteq \mathbf{R}^n$, the Sobolev space $W^{1,n}(D)$ continuously embeds into $BMO^\infty(D)$.

\section{Boundedness of the Helmholtz projection in $vBMOL^2$} \label{sec:3}
Let $\Omega \subset \mathbf{R}^n$ be a uniformly $C^3$ domain of type $(\alpha,\beta,K)$. We consider $f \in vBMOL^2(\Omega)$. Within this whole chapter, we take
\[
\varepsilon < \mathrm{min} \bigg\{ \frac{\beta}{96}, \frac{c_{1/2}^\Omega}{7}, \frac{M_0}{12nK}, \frac{1}{1152nK + 48},  \frac{c_0(n,K)}{12} \bigg\},
\]
where the constant $M_0$ will be defined in the proof of Lemma \ref{FBL} in Section \ref{sub:EUB} and the constant $c_0(n,K)$ is defined in Lemma \ref{EBP}.
\subsection{The $BMO-L^r$ interpolation in domain} \label{sub:IHT}
Let us firstly recall a dyadic Whitney decomposition of an open domain $D$ in $\mathbf{R}^n$.
Let $\mathcal{D}=\{Q_j\}_{j\in\mathbf{N}}$ be a set of dyadic closed cubes with side length $\ell(Q_j)$ contained in $D$ satisfying following four conditions.
\begin{enumerate}
\item[(i)] $D = \cup_j Q_j$,  
\item[(i\hspace{-0.1em}i)] $\operatorname{int}Q_j \cap \operatorname{int}Q_k = \emptyset$ if $j=k$,  
\item[(i\hspace{-0.1em}i\hspace{-0.1em}i)] $\sqrt{n} \leq d(Q_j, \mathbf{R}^n \setminus D) / \ell(Q_j) \leq 4\sqrt{n}$ for all $j \in \mathbf{N}$,  
\item[(i\hspace{-0.1em}v)] $1/4 \leq \ell(Q_k) / \ell(Q_j) \leq 4$ if $Q_j \cap Q_k \neq \emptyset$.
\end{enumerate}
We say that $\mathcal{D}$ is called a Whitney decomposition of $D$.
Such a decomposition exists for any open domain $D$; see e.g. \cite[Appendix J]{GraC}.
 Here $d(A,B)$ for sets $A,B$ in $\mathbf{R}^n$ is defined as $d(A,B) = \operatorname{inf} \left\{ |x-y| \bigm| x \in A,\ y \in B \right\}$.

There are two important distance functions on $\mathcal{D}$.
For $Q_j,Q_k \in \mathcal{D}$, a family $\left\{ Q(\ell) \right\}^m_{\ell=0} \subset \mathcal{D}$ is called a Whitney chain of length $m$ if $Q(0)=Q_j$ and $Q(m)=Q_k$ such that $Q(\ell) \cap Q(\ell+1) \neq \emptyset$ for $\ell$ with $0 \leq \ell \leq m-1$.
Then the length of the shortest Whitney chain connecting $Q_j$ and $Q_k$ gives a distance on $\mathcal{D}$, which is denoted by $d_1(Q_j,Q_k)$.
The second distance for $Q_j,Q_k \in \mathcal{D}$ is defined as
\[
	d_2(Q_j,Q_k) := \log \left| \frac{\ell(Q_j)}{\ell(Q_k)} \right| + \log \left| \frac{\ell(Q_j,Q_k)}{\ell(Q_j) + \ell(Q_k)} + 1 \right|.
\]
Note that $d_1$ and $d_2$ are invariant under dilation as well as translation and rotation; see e.g. \cite{PJB}.

We then recall that a domain $D$ is called a uniform domain if there exists constants $a,b > 0$ such that for all $x,y \in D$, there exists a rectifiable curve $\gamma \subset D$ of length $\ell(\gamma) \leq a |x-y|$ with $\mathrm{min} \{ \ell\big( \gamma(x,z) \big), \ell\big( \gamma(y,z) \big) \} \leq b d(z, \partial D)$, where $\gamma(x,z)$ denotes the part of $\gamma$ between $x$ and $z$ and $\gamma(y,z)$ denotes the part between $y$ and $z$; see e.g. \cite{GO}, \cite{MaSa}. A domain $D$ is a uniform domain is equivalent to the existence of a constant $K_\ast>0$ such that
\begin{align} \label{CCUD}
d_1(Q_j,Q_k) \leq K_\ast d_2(Q_j,Q_k)
\end{align}
holds for any $Q_j, Q_k \in \mathcal{D}$ where $\mathcal{D}$ is a Whitney decomposition for $D$. Jones \cite{PJB} proves that there exists a bounded linear extension operator for $BMO^\infty(D)$ if and only if $D$ is a uniform domain. 
\begin{proposition} \label{ETBL}
Let $D$ be a uniform domain in $\mathbf{R}^n$. Let $1 \leq r < \infty$. For any $\delta>0$, there exists a constant $C=C(K_\ast, \delta)>0$ such that for any $g \in BMOL^r(D)$, there is a linear extension $\overline{g} \in BMOL^r(\mathbf{R}^n)$ such that
\[
\| \overline{g} \|_{BMOL^r(\mathbf{R}^n)} \leq C(K_\ast, \delta) \| g \|_{BMOL^r(D)}
\]
and $\operatorname{supp} \overline{g} \subseteq \overline{D_\delta}$ where $D_\delta := \{ x \in \mathbf{R}^n \bigm| d(x, D) < \delta\}$.
\end{proposition}

This proposition is analogous to \cite[Theorem 12]{GigaGuPA}, which proves that in the case where $D$ is a uniform domain, for any $g$ belongs to $bmo_\infty^\infty(D)$ and $\delta>0$, there is an linear extension $\overline{g} \in bmo(\mathbf{R}^n)$ such that
\[
\| \overline{g} \|_{bmo(\mathbf{R}^n)} \leq C(K_\ast,\delta) \| g \|_{bmo_\infty^\infty(D)}
\]
and $\operatorname{supp} \overline{g} \subseteq \overline{D_\delta}$.
Here $bmo_\infty^\infty(D) := BMO^\infty(D) \cap L_{\mathrm{ul}}^1(D)$ denotes the local $BMO$ space where
\[
L^1_{\mathrm{ul}}(D) := \left\{ g \in L^1_{\mathrm{loc}}(D) \biggm| \| g \|_{L^1_{\mathrm{ul}}(D)} := \sup_{x \in \mathbf{R}^n} \int_{B_1(x) \cap D} \bigl| g(y) \bigr| \, dy < \infty \right\}
\]
and $bmo(\mathbf{R}^n) := BMO(\mathbf{R}^n) \cap L_{\mathrm{ul}}^1(\mathbf{R}^n)$. Proposition \ref{ETBL} can be proved by following similar idea of the proof of \cite[Theorem 12]{GigaGuPA}, which is basically the idea of Jones extension \cite{PJB}. For the completeness of the story, we shall give a sketch of the proof of Proposition \ref{ETBL} here.

\begin{proof}[Proof of Proposition \ref{ETBL}]
Fix $\delta>0$ and $g \in BMOL^r(D)$. We set $k_\delta$ to be the smallest integer such that $2^{-k_\delta} < \frac{\delta}{5\sqrt{n}}$.
Let $\mathcal{D} = \{ Q_j \}_{j \in \mathbf{N}}$ be a Whitney decomposition of $D$ and $\mathcal{D}' = \{ Q_j' \}_{j \in \mathbf{N}}$ be a Whitney decomposition of $D^{\mathrm{c}}$.
Let $\mathcal{D}_\ast$ be the set of Whitney cubes in $\mathcal{D}$ whose side length is strictly greater than $2^{-k_\delta}$. For each $Q_i \in \mathcal{D}_\ast$, we define a function $g_i$ on $D$ by
\begin{eqnarray*}
	g_i(x) := \left \{
\begin{array}{ll}
	g_{Q_i} & \mathrm{if} \; \, x \in Q_i, \\
	0 & \mathrm{else}
\end{array}
\right.
\end{eqnarray*}
where $g_{Q_i} = \frac{1}{|Q_i|} \int_{Q_i} g(y) \, dy$ for each $Q_i \in \mathcal{D}_\ast$. 
We further define a function $h$ on $\mathbf{R}^n$ by $h^\ast := \sum_{Q_i \in \mathcal{D}_\ast} g_i$.
Note that for any $Q_i \in \mathcal{D}_\ast$, we have that $\ell(Q_i) > 2^{-k_\delta} \geq \frac{\delta}{10 \sqrt{n}}$. Hence, by H$\ddot{\text{o}}$lder's inequality we deduce that
\[
\| h^\ast \|_{L^\infty(\mathbf{R}^n)} \leq \sum_{Q_i \in \mathcal{D}_\ast} \big| g_{Q_i} \big| \leq \sum_{Q_i \in \mathcal{D}_\ast} \frac{1}{|Q_i|^{\frac{1}{r}}} \cdot \| g \|_{L^r(Q_i)} \leq \frac{C(r,n)}{\delta^{\frac{n}{r}}} \| g \|_{L^r(D)}
\]
and
\[
\| h^\ast \|_{L^r(\mathbf{R}^n)} \leq \sum_{Q_i \in \mathcal{D}_\ast} \| g_i \|_{L^r(Q_i)} \leq \sum_{Q_i \in \mathcal{D}_\ast} |Q_i|^{\frac{1}{r}} \cdot \big| g_{Q_i} \big| \leq \sum_{Q_i \in \mathcal{D}_\ast} \| g \|_{L^r(Q_i)} \leq \| g \|_{L^r(D)},
\]
i.e., we have that $h^\ast \in BMOL^r(\mathbf{R}^n)$.

Let $g^\ast := g - h^\ast \in BMOL^r(D)$. We do Jones extension to $g^\ast$. If $D$ is unbounded, for each $Q_j^{'} \in \mathcal{D}^{'}$, we find a nearset $Q_j \in \mathcal{D}$ satisfying $\ell(Q_j) \geq \ell(Q_j^{'})$. 
We define that $\widetilde{g^\ast} = g^\ast$ in $D$ and $\widetilde{g^\ast}(x) = g^\ast_{Q_j}$ for $x \in Q_j^{'}$. 
If $D$ is bounded, we pick $Q_0 \in \mathcal{D}$ such that $\ell(Q_0) = \underset{Q_j \in \mathcal{D}}{\mathrm{sup}} \ell(Q_j)$. 
We define that $\widetilde{g^\ast} = g^\ast$ in $D$, $\widetilde{g^\ast}(x) = g^\ast_{Q_j}$ for $x \in Q_j^{'}$ if $\ell(Q_j^{'}) \leq \ell(Q_0)$ where $Q_j$ in this case is the nearest Whitney cube in $\mathcal{D}$ to $Q_j'$ which satisfies $\ell(Q_j) \geq \ell(Q_j')$
and $\widetilde{g^\ast}(x) = g^\ast_{Q_0}$ for $x \in Q_j^{'}$ if $\ell(Q_j') > \ell(Q_0)$. By Jones \cite{PJB}, we have that $\widetilde{g^\ast} \in BMO(\mathbf{R}^n)$ and $[\widetilde{g^\ast}]_{BMO(\mathbf{R}^n)} \leq C(K_\ast)[g^\ast]_{BMO^\infty(D)}$.
By our construction, we can easily see that $\operatorname{supp} \widetilde{g^\ast} \subseteq \overline{D_\delta}$. It is sufficient to establish the $L^r$ estimate for $\widetilde{g^\ast}$.

For $Q_j \in \mathcal{D}$, we define $N'(Q_j) \in \mathbf{N}$ to be the number of $Q_j'$ in $\mathcal{D}'$ such that we can pick $Q_j$ as the nearest Whitney cube in $\mathcal{D}$ to these $Q_j'$ which satisfies $\ell(Q_j) \geq \ell(Q_j')$. We can show that
\[
\sup_{Q_j \in \mathcal{D}} N'(Q_j) < \infty.
\]
Let $Q_j' \in \mathcal{D}'$ and $Q_j$ be a nearest cube to $Q_j'$ in $\mathcal{D}$ which satisfies $\ell(Q_j) \geq \ell(Q_j')$. Let the center of $Q_j$ be denoted by $x_j$. 
For any cube $Q$, let the diameter of $Q$ be denoted by $\mathrm{diam}(Q)$, i.e., $\mathrm{diam}(Q) := \sup_{x,y \in \overline{Q}} \big| x-y \big|$.
A simple triangle inequality implies that $Q_j' \subset B_{\mathrm{diam}(Q_j) + \mathrm{diam}(Q_j') + d(Q_j,Q_j')}(x_j)$.
\cite[Lemma 2.10]{PJB} guarantees that if $Q_j' \in \mathcal{D}'$ and $Q_j$ is a nearest cube to $Q_j'$ in $\mathcal{D}$ which satisfies $\ell(Q_j) \geq \ell(Q_j')$, then we must have $d(Q_j,Q_j') \leq 65 K_\ast^2 \cdot \ell(Q_j')$. Hence, we may deduce that
\[
\mathrm{diam}(Q_j) + \mathrm{diam}(Q_j') + d(Q_j,Q_j') \leq 67 K_\ast^2 \cdot \ell(Q_j)
\]
as we may assume that $K_\ast > n$ without loss of generality.
Suppose that $Q_j'' \in \mathcal{D}'$ is another Whitney cube such that we can pick $Q_j$ as a nearest cube to $Q_j''$ in $\mathcal{D}$. Then, we must have that $Q_j' \cup Q_j'' \subset B_{67 K_\ast^2 \cdot \ell(Q_j)}(x_j)$.
Note that $Q_j$ is a nearest cube in $\mathcal{D}$ to both $Q_j'$ and $Q_j''$ implies that
\[
\ell(Q_j') \leq \ell(Q_j) \leq 2 \ell(Q_j'), \quad \ell(Q_j'') \leq \ell(Q_j) \leq 2 \ell(Q_j'').
\]
Since $\mathcal{D}$ and $\mathcal{D}'$ are both dyadic Whitney decompositions, $\ell(Q_j')$ and $\ell(Q_j'')$ can only be either $\ell(Q_j)/2$ or $\ell(Q_j)$.
Since $Q_j'$ and $Q_j''$ have disjoint interior, we deduce that
\[
N'(Q_j) \leq (2^n +1) \cdot \frac{\big| B_{67 K_\ast^2 \cdot \ell(Q_j)}(x_j) \big|}{\ell(Q_j)^n} \leq (2^n+1) 67^n K_\ast^{2n} < \infty.
\]

The $L^r$ estimate for $\widetilde{g^\ast}$ can be derived as follows. Since $\widetilde{g^\ast} = 0$ in $Q_j' \in \mathcal{D}'$ such that $\ell(Q_j') > 2^{-k_\delta}$, we have that
\[
\| \widetilde{g^\ast} \|_{L^r(\mathbf{R}^n)}^r = \sum_{Q_j \in \mathcal{D}} \int_{Q_j} \big| \widetilde{g^\ast} \big|^r \, dy + \sum_{Q_j' \in \mathcal{D}'} \int_{Q_j'} \big| \widetilde{g^\ast} \big|^r \, dy \leq \| g^\ast \|_{L^r(D)}^r + \sum_{\ell(Q_j') \leq 2^{-k_\delta}} \int_{Q_j'} \big| \widetilde{g^\ast} \big|^r \, dy.
\]
Since 
\[
\int_{Q_j'} \big| \widetilde{g^\ast} \big|^r \, dy = |Q_j'| \cdot \big| g^\ast_{Q_j} \big|^r \leq \frac{|Q_j'|}{|Q_j|} \cdot \| g^\ast \|_{L^r(Q_j)}^r \leq \| g^\ast \|_{L^r(Q_j)}^r
\]
where $Q_j$ is a nearest Whitney cube in $\mathcal{D}$ to $Q_j'$ that satisfies $\ell(Q_j) \geq \ell(Q_j')$, we have that
\begin{align} \label{L2Eg}
\sum_{\ell(Q_j') \leq 2^{-k_\delta}} \int_{Q_j'} \big| \widetilde{g^\ast} \big|^r \, dy \leq \sum_{\ell(Q_j') \leq 2^{-k_\delta}} \| g^\ast \|_{L^r(Q_j)}^r.
\end{align}
Since for any $Q_j \in \mathcal{D}$, there exist at most $(2^n+1) 67^n K_\ast^{2n}$ Whitney cubes $Q_j'$ in $\mathcal{D}'$ such that we can pick $Q_j$ as a nearest cube in $\mathcal{D}$ to $Q_j'$ which satisfies $\ell(Q_j) \geq \ell(Q_j')$, each $\| g^\ast \|_{L^r(Q_j)}^r$ appears in the right hand side of (\ref{L2Eg})
for at most $(2^n+1) 67^n K_\ast^{2n}$ times. Therefore,
\[
\sum_{\ell(Q_j') \leq 2^{-k_\delta}} \| g^\ast \|_{L^r(Q_j)}^r \leq (2^n+1) 67^n K_\ast^{2n} \sum_{\ell(Q_j) \leq 2^{-k_\delta+1}} \| g^\ast \|_{L^r(Q_j)}^r \leq (2^n+1) 67^n K_\ast^{2n} \| g^\ast \|_{L^r(D)}^r.
\]
Finally, by setting $\overline{g} = \widetilde{g^\ast} + h^\ast$, we obtain our desired extension for $g$.
\end{proof}
Now we are ready to establish the $BMO-L^r$ interpolation inequality for domain cases. 
\begin{proposition} \label{ITP}
Let $D$ be either a uniformly $C^2$ domain or a uniform domain in $\mathbf{R}^n$. There exists a constant $C>0$, independent of $u$, $q$ and $r$, such that the estimate
\begin{align} \label{ITPE}
\| u \|_{L^q(D)} \leq C q \| u \|_{BMOL^r(D)}
\end{align}
holds for all $u \in BMOL^r(D)$ with $1 \leq r < \infty$ and for all $q$ with $r \leq q < \infty$. In the case where $D$ is a uniformly $C^2$ domain, the constant $C=C(\alpha_D, \beta_D, K_D)$ depends only on the boundary regularity of $\partial D$. While in the case where $D$ is a uniform domain, the constant $C=C(K_\ast)$ depends only on the constant $K_\ast$.
\end{proposition}
\begin{proof}
Let $1 \leq r < \infty$ and $r \leq q < \infty$. 
Suppose that $D$ is a uniformly $C^2$ domain of type $(\alpha_D, \beta_D, K_D)$. By \cite[Theorem 1]{Gu}, we see that there exists $c^\ast_D>0$ such that for any $\rho \in (0,c^\ast_D)$ and $u \in BMOL^r(D)$, there is an extension $\widetilde{u} \in BMOL^r(\mathbf{R}^n)$ such that
\[
\| \widetilde{u} \|_{BMOL^r(\mathbf{R}^n)} \leq C(\alpha_D, \beta_D, K_D, \rho) \| u \|_{BMOL^r(D)}.
\]
By the whole space $BMO-L^r$ interpolation inequality \cite[Theorem 2.2]{KW}, we deduce that
\begin{align*}
\| u \|_{L^q(D)} &\leq \| \widetilde{u} \|_{L^q(\mathbf{R}^n)} \leq C(n) q \| \widetilde{u} \|_{L^r(\mathbf{R}^n)}^\frac{r}{q} \| \widetilde{u} \|_{BMO(\mathbf{R}^n)}^{1 - \frac{r}{q}} \\ 
&\leq C(n)q \| \widetilde{u} \|_{BMOL^r(\mathbf{R}^n)} \leq C(\alpha_D, \beta_D, K_D) q \| u \|_{BMOL^r(D)}.
\end{align*}

Suppose that $D$ is a uniform domain. By Proposition \ref{ETBL}, we see that for any $u \in BMOL^r(D)$, there exists a linear extension $\overline{u} \in BMOL^r(\mathbf{R}^n)$ such that
\[
\| \overline{u} \|_{BMOL^r(\mathbf{R}^n)} \leq C(K_\ast) \| u \|_{BMOL^r(D)}
\]
where $K_\ast$ is the constant that characterize the uniform domain $D$. Again, by applying the whole space $BMO-L^r$ interpolation inequality \cite[Theorem 2.2]{KW}, we obtain that
\[
\| u \|_{L^q(D)} \leq C(n) q \| \overline{u} \|_{L^r(\mathbf{R}^n)}^\frac{r}{q} \| \overline{u} \|_{BMO(\mathbf{R}^n)}^{1 - \frac{r}{q}} \leq C(n) q \| \overline{u} \|_{BMOL^r(\mathbf{R}^n)} \leq C(K_\ast, n) q \| u \|_{BMOL^r(D)}.
\] 
\end{proof}

By Proposition \ref{ITP}, $f \in vBMOL^2(\Omega)$ implies that $f \in \widetilde{L}^{2n}(\Omega) := L^{2n}(\Omega)^n \cap L^2(\Omega)^n$. Since $\widetilde{L}^{2n}(\Omega)$ admits Helmholtz decomposition, see \cite[Theorem 2.1]{FKS} and \cite[Theorem 1.2]{FKSH}, $f$ admits the decomposition
\[
f = f_0 + \nabla p, \quad f_0 \in \widetilde{L}^{2n}_\sigma(\Omega), \quad \nabla p \in \widetilde{G}^{2n}(\Omega)
\]
with the estimate
\begin{align} \label{HE2N}
\| f_0 \|_{\widetilde{L}^{2n}(\Omega)} + \| \nabla p \|_{\widetilde{L}^{2n}(\Omega)} \leq C \| f \|_{\widetilde{L}^{2n}(\Omega)}, \quad C = C(\alpha,\beta,K)>0.
\end{align}
In the following sections, we shall prove that this Helmholtz decomposition of $f$ in $\widetilde{L}^{2n}(\Omega)$, i.e., $f = f_0 + \nabla p$, is indeed the Helmholtz decomposition of $f$ in $vBMOL^2(\Omega)$.

\subsection{Interior $BMO$ estimate} \label{sub:IE}
Let $x \in \overline{\Omega_{3 \varepsilon}}$ be a random point and $\varphi_x \in C_{\mathrm{c}}^\infty(\mathbf{R}^n)$ be the cut-off function defined in Proposition \ref{CFB}.
We then follow the idea of \cite{FKS} and \cite{FKSH} and consider the local equation
\[
\varphi_x f = \varphi_x f_0 + \nabla \big( \varphi_x \big( p - M_x) \big) - \nabla \varphi_x (p-M_x), \quad M_x := \frac{1}{|B_{2 \varepsilon}(x)|} \int_{B_{2 \varepsilon}(x)} p \, dy.
\]

Since $C_{\mathrm{c}}^\infty(\overline{\Omega}) \subset G^{2n}(\Omega)$, $f_0 \in L^{2n}_\sigma(\Omega)$ implies that $f_0 \in L^{2n}(\Omega)^n$ with $\operatorname{div} f_0 = 0$ in $\Omega$ and $f_0 \cdot \mathbf{n} = 0$ on $\Gamma$. Since $\operatorname{supp} \varphi_x \subseteq \overline{B_{\frac{3 \varepsilon}{2}}(x)}$, by considering Green's formula in $B_{2 \varepsilon}(x)$, see e.g. \cite[Sec. II.1.2]{HSo}, we see that $\nabla \varphi_x \cdot f_0 \in L_0^{2n}\big( B_{2 \varepsilon}(x) \big) := \{ g \in L^{2n}\big( B_{2 \varepsilon}(x) \big) \bigm| \int_{B_{2 \varepsilon}(x)} g \, dy = 0 \}$. 
Let $B_{2n}^x : L_0^{2n}\big( B_{2 \varepsilon}(x) \big) \to W_0^{1,2n}\big( B_{2 \varepsilon}(x) \big)^n$ be the Bogovski$\breve{\i}$ operator which solves the divergence problem 
\[
\operatorname{div} u = g, \quad u \bigm|_{\partial B_{2 \varepsilon}(x)} = 0
\]
for $g \in L_0^{2n}\big( B_{2 \varepsilon}(x) \big)$. We then set $\omega_x := B_{2n}^x\big( \nabla \varphi_x \cdot f_0 \big)$. Since the open ball $B_{2 \varepsilon}(x)$ is obviously star-like with respect to the open ball $B_\varepsilon(x)$, by estimate (\ref{BGOE}) we have that
\begin{align} \label{BGOx}
\| \omega_x \|_{W^{1,2n}\big( B_{2 \varepsilon}(x) \big)} \leq C(n) \| \nabla \varphi_x \cdot f_0 \|_{L^{2n}\big( B_{2 \varepsilon}(x) \big)}.
\end{align}

Since $\overline{B_{2 \varepsilon}(x)} \subset \Omega$, the definition of $\widetilde{G}^{2n}(\Omega)$ ensures that $p$ belongs to the Sobolev space $W^{1,2}\big( B_{2 \varepsilon}(x) \big) \cap W^{1,2n}\big( B_{2 \varepsilon}(x) \big)$. By considering the $L^q$-interpolation, we see that $p \in W^{1,n}\big( B_{2 \varepsilon}(x) \big)$. Since $W^{1,n}\big( B_{2 \varepsilon}(x) \big) \hookrightarrow BMO^\infty\big( B_{2 \varepsilon}(x) \big)$, we indeed have that $p \in BMOL^2\big( B_{2 \varepsilon}(x) \big)$.

\begin{lemma} \label{FIL}
Let $x \in \overline{\Omega_{3 \varepsilon}}$.
There exists a constant $C_B>0$, which is independent of $x \in \overline{\Omega_{3 \varepsilon}}$, such that
\begin{equation} \label{FILE}
\begin{split}
&\| \varphi_x f_0 \|_{vBMOL^2\big( B_{2 \varepsilon}(x) \big)} + \| \varphi_x \nabla p \|_{vBMOL^2\big( B_{2 \varepsilon}(x) \big)} \\
&\ \ \leq C_B \bigg( \| \varphi_x f \|_{vBMOL^2\big( B_{2 \varepsilon}(x) \big)} + 2 \| \nabla \varphi_x (p - M_x) \|_{vBMOL^2\big( B_{2 \varepsilon}(x) \big)} + 2 \| \omega_x \|_{vBMOL^2\big( B_{2 \varepsilon}(x) \big)} \bigg).
\end{split}
\end{equation}
\end{lemma}
\begin{proof}
Similar to the discussion in Section \ref{sub:BMP}, by considering Jones extension for Sobolev spaces \cite{PJS}, we can see that Morrey's inequality (\ref{MOR}) also holds for domain $B_{2 \varepsilon}(x)$, i.e., the Sobolev space $W^{1,2n}\big( B_{2 \varepsilon}(x) \big)$ is continuously embedded into the H$\ddot{\text{o}}$lder space $C^{0,\frac{1}{2}}\big( B_{2 \varepsilon}(x) \big)$ with the estimate
\begin{align} \label{MORB}
\| g \|_{C^{0,\frac{1}{2}}\big( B_{2 \varepsilon}(x) \big)} \leq \frac{C(n)}{\varepsilon} \| g \|_{W^{1,2n}\big( B_{2 \varepsilon}(x) \big)}
\end{align}
holds for any $g \in W^{1,2n}\big( B_{2 \varepsilon}(x) \big)$. Hence, we have that $\omega_x = B_{2n}^x\big( \nabla \varphi_x \cdot f_0 \big) \in L^\infty\big( B_{2 \varepsilon}(x) \big)^n$. Since $\operatorname{supp} \varphi_x \subseteq \overline{B_{\frac{3 \varepsilon}{2}}(x)}$, we see that $\varphi_x f + \nabla \varphi_x (p - M_x) - \omega_x \in vBMOL^2\big( B_{2 \varepsilon}(x) \big)$.

Since $B_{2 \varepsilon}(x)$ is a bounded smooth domain, the Helmholtz decomposition of $vBMOL^2\big( B_{2 \varepsilon}(x) \big)$ holds, see \cite[Theorem 1]{GigaGuMA}. There exist $f_{0,\ast} \in vBMOL_\sigma^2\big( B_{2 \varepsilon}(x) \big)$ and $\nabla p_\ast \in GvBMOL^2\big( B_{2 \varepsilon}(x) \big)$ such that $\varphi_x f + \nabla \varphi_x (p - M_x) - \omega_x = f_{0,\ast} + \nabla p_\ast$ where
\begin{align*}
vBMOL_\sigma^2\big( B_{2 \varepsilon}(x) \big) &= \{ v \in vBMOL^2\big( B_{2 \varepsilon}(x) \big) \bigm| \operatorname{div} v = 0 \;\, \text{in} \;\, B_{2 \varepsilon}(x), v \cdot \mathbf{n} = 0 \;\, \text{on} \;\, \partial B_{2 \varepsilon}(x) \}, \\
GvBMOL^2\big( B_{2 \varepsilon}(x) \big) &= \big\{ \nabla h \in vBMOL^2\big( B_{2 \varepsilon}(x) \big) \bigm| h \in L^\infty\big( B_{2 \varepsilon}(x) \big) \big\}.
\end{align*}
By Proposition \ref{ITP}, we have that $f_{0,\ast}, \nabla p_\ast \in \widetilde{L}^{2n}\big( B_{2 \varepsilon}(x) \big)$. Note that for any bounded $C^1$ domain $D$, we have that
\[
\widetilde{L}_\sigma^{2n}(D) = \{ v \in \widetilde{L}^{2n}(D) \bigm| \operatorname{div} v = 0 \;\, \text{in} \;\, B_{2 \varepsilon}(x), v \cdot \mathbf{n} = 0 \;\, \text{on} \;\, \partial B_{2 \varepsilon}(x) \};
\]
see e.g. \cite{FKSH}, \cite{SiSo}. Hence, we see that $f_{0,\ast} \in \widetilde{L}_\sigma^{2n}\big( B_{2 \varepsilon}(x) \big)$ and $\nabla p_\ast \in \widetilde{G}^{2n}\big( B_{2 \varepsilon}(x) \big)$, i.e., $\varphi_x f + \nabla \varphi_x (p - M_x) - \omega_x = f_{0,\ast} + \nabla p_\ast$ is indeed the Helmholtz decomposition of $\varphi_x f + \nabla \varphi_x (p - M_x) - \omega_x$ in $\widetilde{L}^{2n}\big( B_{2 \varepsilon}(x) \big)$. Since the Helmholtz decomposition of $\widetilde{L}^{2n}(D)$ is unique for any uniformly $C^1$ domain $D$ \cite[Theorem 1.2]{FKSH}, we conclude that $f_{0,\ast} = \varphi_x f_0 - \omega_x$ and $\nabla p_\ast = \nabla \big( \varphi_x (p - M_x) \big)$. That means that
\[
\varphi_x f + \nabla \varphi_x (p - M_x) - \omega_x = \big( \varphi_x f_0 - \omega_x \big) + \nabla \big( \varphi_x (p - M_x) \big)
\]
is indeed the Helmholtz decomposition of $\varphi_x f + \nabla \varphi_x \cdot (p - M_x) - \omega_x$ in $vBMOL^2\big( B_{2 \varepsilon}(x) \big)$.

Note that for any domain $D \subset \mathbf{R}^n$, the $vBMOL^2$-norm is translational invariant in the sense that for any $y \in \mathbf{R}^n$, the estimate
\[
\| g \|_{vBMOL^2(D)} = \| g^{-y} \|_{vBMOL^2(D+y)}
\] 
holds for any $g \in vBMOL^2(D)$ with $g^{-y}(\cdot) := g(\cdot -y)$ and $D+y := \{ x+y \bigm| x \in D\}$.
Therefore, there exists a constant $C_B>0$, which is independent of $x \in \overline{\Omega_{3 \varepsilon}}$, such that
\begin{align*}
&\| \varphi_x f_0 - \omega_x \|_{vBMOL^2\big( B_{2 \varepsilon}(x) \big)} + \| \nabla \big( \varphi_x (p - M_x) \big) \|_{vBMOL^2\big( B_{2 \varepsilon}(x) \big)} \\
&\ \ \leq C_B \| \varphi_x f + \nabla \varphi_x (p - M_x) - \omega_x \|_{vBMOL^2\big( B_{2 \varepsilon}(x) \big)}.
\end{align*}
By the triangle inequality, we obtain estimate (\ref{FILE}).
\end{proof}

In order to estimate the right hand side of estimate (\ref{FILE}), we need the following multiplication rule which enables us to estimate products of H$\ddot{\text{o}}$lder continuous functions and $BMOL^r$ functions defined in a domain.

\begin{proposition} \label{BLRM}
Let $D \subset \mathbf{R}^n$ be a uniform domain and $r \in [1,\infty)$. There exists a constant $C=C(K_\ast)>0$, depending only on the constant $K_\ast$ in (\ref{CCUD}) which characterizes the uniform domain $D$, such that the estimate
\begin{align} \label{MRUD}
\| \varphi_D v \|_{BMOL^r(D)} \leq C(K_\ast) \| \varphi_D \|_{C^\gamma(D)} \| v \|_{BMOL^r(D)}
\end{align}
holds for any $v \in BMOL^r(D)$ and $\varphi_D \in C^\gamma(D)$ where $\gamma \in (0,1)$.
\end{proposition}
\begin{proof}
Let $\varphi_D \in C^\gamma(D)$ with $\gamma \in (0,1)$, there exists an linear extension $\overline{\varphi_D} \in C^\gamma(\mathbf{R}^n)$ such that the restriction of $\overline{\varphi_D}$ in $D$ equals $\varphi_D$ and $\| \overline{\varphi_D} \|_{C^\gamma(\mathbf{R}^n)} \leq \| \varphi_D \|_{C^\gamma(D)}$; see e.g. \cite[Theorem 13]{GigaGuPA}.
Let $\overline{v} \in BMOL^r(\mathbf{R}^n)$ be the Jones extension of $v \in BMOL^r(D)$ constructed in Proposition \ref{ETBL}.
Since the multiplication of $\varphi_\ast \in C^\gamma(\mathbf{R}^n)$ to $g \in BMOL^r(\mathbf{R}^n)$ is bounded \cite{GigaGuPA}, by Proposition \ref{ETBL} we deduce that
\begin{align*}
\| \overline{\varphi_D} \cdot \overline{v} \|_{BMOL^r(\mathbf{R}^n)} \leq C \| \overline{\varphi_D} \|_{C^\gamma(\mathbf{R}^n)} \| \overline{v} \|_{BMOL^r(\mathbf{R}^n)} \leq C(K_\ast) \| \varphi_D \|_{C^\gamma(D)} \| v \|_{BMOL^r(D)}.
\end{align*}
\end{proof}

By considering estimate (\ref{FILE}) for every $x \in \overline{\Omega_{3 \varepsilon}}$, we can deduce an interior $BMO^\varepsilon$-estimate for $f_0$ and $\nabla p$ in $\Omega_{2 \varepsilon}$. For $\rho \in (0, R_\ast)$, we let $\partial \Omega_\rho := \{ x \in \Omega \bigm| d_\Gamma(x) = \rho \}$.

\begin{lemma} \label{FPBI}
The estimate
\begin{align} \label{IE}
[f_0]_{BMO^\varepsilon(\Omega_{2 \varepsilon})} + [\nabla p]_{BMO^\varepsilon(\Omega_{2 \varepsilon})} \leq \frac{C}{\varepsilon^2} \| f \|_{vBMOL^2(\Omega)}.
\end{align}
holds with $C=C(\alpha,\beta,K)>0$.
\end{lemma}
\begin{proof}
Let $x \in \Omega_{2 \varepsilon} \setminus \overline{\Omega_{3 \varepsilon}}$ and $0<r<\varepsilon$ be such that $B_r(x) \subset \Omega_{2 \varepsilon}$. 
Since $d_\Gamma(x) < 3 \varepsilon < R_\ast$, there exists a unique $z_0 \in \Gamma$ such that $d_\Gamma(x) = |x - z_0|$. Let $x_0$ be the unique point in $\partial \Omega_{3 \varepsilon}$ such that $|x_0 - z_0| = 3 \varepsilon$. Since $B_r(x) \subset \Omega_{2 \varepsilon}$, we must have that $|x - x_0| + r \leq \varepsilon$. Thus, we deduce that $B_r(x) \subset B_\varepsilon(x_0)$. Let $\varphi_{x_0} \in C_{\mathrm{c}}^\infty\big( B_{2 \varepsilon}(x_0) \big)$ be defined as in Proposition \ref{CFB}. Since $\varphi_{x_0} = 1$ in $B_\varepsilon(x_0)$, it holds that
\begin{align*}
\frac{1}{|B_r(x)|} \int_{B_r(x)} \big| f_0 - (f_0)_{B_r(x)} \big| \, dy &\leq [f_0]_{BMO^\infty\big( B_\varepsilon(x_0) \big)} \leq \| \varphi_{x_0} f_0 \|_{vBMOL^2\big( B_{2 \varepsilon}(x_0) \big)}, \\
\frac{1}{|B_r(x)|} \int_{B_r(x)} \big| \nabla p - (\nabla p)_{B_r(x)} \big| \, dy &\leq [\nabla p]_{BMO^\infty\big( B_\varepsilon(x_0) \big)} \leq \| \varphi_{x_0} \nabla p \|_{vBMOL^2\big( B_{2 \varepsilon}(x_0) \big)}.
\end{align*}
In the case that $x \in \overline{\Omega_{3 \varepsilon}}$ and $0<r<\varepsilon$, by considering the cut-off function $\varphi_x \in C_{\mathrm{c}}^\infty\big( B_{2 \varepsilon}(x) \big)$ directly, we naturally have that
\begin{align*}
\frac{1}{|B_r(x)|} \int_{B_r(x)} \big| f_0 - (f_0)_{B_r(x)} \big| \, dy &\leq \| \varphi_{x} f_0 \|_{vBMOL^2\big( B_{2 \varepsilon}(x) \big)}, \\
\frac{1}{|B_r(x)|} \int_{B_r(x)} \big| \nabla p - (\nabla p)_{B_r(x)} \big| \, dy &\leq \| \varphi_{x} \nabla p \|_{vBMOL^2\big( B_{2 \varepsilon}(x) \big)}.
\end{align*}
Hence, in order to estimate the $BMO^\varepsilon$-seminorm of $f_0$ and $\nabla p$ in $\Omega_{2 \varepsilon}$, it is sufficient to estimate the right hand side of estimate (\ref{FILE}).

Let $x \in \overline{\Omega_{3 \varepsilon}}$.
Since $\operatorname{supp} \varphi_x \subseteq \overline{B_{\frac{3 \varepsilon}{2}}(x)}$, we have that
\[
\| \varphi_x f \|_{vBMOL^2\big( B_{2 \varepsilon}(x) \big)} = \| \varphi_x f \|_{BMOL^2\big( B_{2 \varepsilon}(x) \big)}
\]
and
\[
\| \nabla \varphi_x (p - M_x) \|_{vBMOL^2\big( B_{2 \varepsilon}(x) \big)} = \| \nabla \varphi_x (p - M_x) \|_{BMOL^2\big( B_{2 \varepsilon}(x) \big)}.
\]
In the case of a bounded Lipschitz domain $D$, the constant $K_\ast$ in (\ref{CCUD}), which characterizes $D$ as a uniform domain, depends only on the Lipschitz regularity of $\partial D$, see e.g. \cite{GigaGuPA}.
Since the Lipschitz regularity of $\partial B_r(y)$ is a universal constant that is independent of $r \in (0,\infty)$ and $y \in \mathbf{R}^n$, by the multiplication rule (\ref{MRUD}) we can deduce that 
\[
\| \varphi_x f \|_{BMOL^2\big( B_{2 \varepsilon}(x) \big)} \leq C \| \varphi_x \|_{C^1\big( B_{2 \varepsilon}(x) \big)} \| f \|_{BMOL^2\big( B_{2 \varepsilon}(x) \big)} \leq \frac{C}{\varepsilon} \| f \|_{vBMOL^2(\Omega)}
\]
and
\begin{align*}
\| \nabla \varphi_x (p - M_x) \|_{BMOL^2\big( B_{2 \varepsilon}(x) \big)} &\leq C \| \nabla \varphi_x \|_{C^1\big( B_{2 \varepsilon}(x) \big)} \| p - M_x \|_{BMOL^2\big( B_{2 \varepsilon}(x) \big)} \\
&\leq \frac{C}{\varepsilon^2} \| p - M_x \|_{BMOL^2\big( B_{2 \varepsilon}(x) \big)}.
\end{align*}
By estimate (\ref{EWBMO}), H$\ddot{\text{o}}$lder's inequality and Poincar$\acute{\text{e}}$ inequality (\ref{POI}), we deduce that
\[
[p - M_x]_{BMO^\infty\big( B_{2 \varepsilon}(x) \big)} \leq C(n) \| \nabla p \|_{L^n\big( B_{2 \varepsilon}(x) \big)} \leq C(n) \| \nabla p \|_{\widetilde{L}^{2n}(\Omega)}
\]
and
\[
\| p - M_x \|_{L^2\big( B_{2 \varepsilon}(x) \big)} \leq C(n) \varepsilon^{n-1} \| \nabla p \|_{L^n\big( B_{2 \varepsilon}(x) \big)} \leq C(n) \| \nabla p \|_{\widetilde{L}^{2n}(\Omega)}.
\]
Hence, we obtain that
\[
\| \nabla \varphi_x (p - M_x) \|_{BMOL^2\big( B_{2 \varepsilon}(x) \big)} \leq \frac{C(n)}{\varepsilon^2} \| \nabla p \|_{\widetilde{L}^{2n}(\Omega)}.
\]
Finally, by Morrey's inequality (\ref{MORB}) and estimate (\ref{BGOx}), we see that
\[
\| \omega_x \|_{vBMOL^2\big( B_{2 \varepsilon}(x) \big)} \leq C(n) \| \omega_x \|_{L^\infty\big( B_{2 \varepsilon}(x) \big)} \leq \frac{C(n)}{\varepsilon} \| \omega_x \|_{W^{1,2n}\big( B_{2 \varepsilon}(x) \big)} \leq \frac{C(n)}{\varepsilon^2} \| f_0 \|_{L^{2n}(\Omega)}.
\]
By estimate (\ref{HE2N}), we obtain Lemma \ref{FPBI}.
\end{proof}

\subsection{$vBMO$ Estimate up to the boundary} \label{sub:EUB}
In this section, we consider the dimension $n \geq 3$.
Let $z_0 \in \Gamma$ be a random point and $\varphi_{z_0} \in C^2(\Omega)$ be the cut-off function defined in Proposition \ref{CFI}.
We treat $\varphi_{z_0}$ as a $C^2$ function defined in $B^\Omega_{12 \varepsilon}(z_0)$ and then consider the local equation
\[
\varphi_{z_0} f = \varphi_{z_0} f_0 + \nabla \big( \varphi_{z_0} \big( p - M_{z_0}) \big) - \big( \nabla \varphi_{z_0} \big) (p-M_{z_0}), \quad M_{z_0} := \frac{1}{|B_{12 \varepsilon}^\Omega(z_0)|} \int_{B_{12 \varepsilon}^\Omega(z_0)} p \, dy.
\]

Since $\operatorname{supp} \varphi_{z_0} \subset \overline{U_{4 \varepsilon}(z_0)} \subset \overline{B^\Omega_{12 \varepsilon}(z_0)}$, by considering Green's formula in $B^\Omega_{12 \varepsilon}(z_0)$, see e.g. \cite[Sec. II.1.2]{HSo}, we have that $\nabla \varphi_{z_0} \cdot f_0 \in L_0^{2n}\big( B^\Omega_{12 \varepsilon}(z_0) \big) := \{ g \in L^{2n}\big( B^\Omega_{12 \varepsilon}(z_0) \big) \bigm| \int_{B^\Omega_{12 \varepsilon}(z_0)} g \, dy = 0 \}$.
Let $B_{2n}^{z_0} : L_0^{2n}\big( B^\Omega_{12 \varepsilon}(z_0) \big) \to W_0^{1,2n}\big( B^\Omega_{12 \varepsilon}(z_0) \big)^n$ be the Bogovski$\breve{\i}$ operator which solves the divergence problem 
\[
\operatorname{div} u = g, \quad u \bigm|_{\partial B^\Omega_{12 \varepsilon}(z_0)} = 0
\]
for $g \in L_0^{2n}\big( B^\Omega_{12 \varepsilon}(z_0) \big)$. 
We then set $\omega_{z_0} := B_{2n}^{z_0}\big( \nabla \varphi_{z_0} \cdot f_0 \big)$. 
By Lemma \ref{BOS}, there exists $x_0 \in B^\Omega_{12 \varepsilon}(z_0)$ such that $B^\Omega_{12 \varepsilon}(z_0)$ is star-like with respect to $B_{3 \varepsilon}(x_0) \subset B^\Omega_{12 \varepsilon}(z_0)$.
Hence, by estimate (\ref{BGOE}) we have that
\begin{align} \label{BGOz}
\| \omega_{z_0} \|_{W^{1,2n}\big( B^\Omega_{12 \varepsilon}(z_0) \big)} \leq C(n) \| \nabla \varphi_{z_0} \cdot f_0 \|_{L^{2n}\big( B^\Omega_{12 \varepsilon}(z_0) \big)}.
\end{align}

Since the boundary of $B^\Omega_{12 \varepsilon}(z_0)$ is locally a continuous function, $p \in \widetilde{G}^{2n}(\Omega)$ ensures that $p \in W^{1,2}\big( B^\Omega_{12 \varepsilon}(z_0) \big) \cap W^{1,2n}\big( B^\Omega_{12 \varepsilon}(z_0) \big)$, see e.g. \cite[\S 2, Th. 7.6]{Nec}.
The $L^q$-interpolation further implies that $p \in W^{1,n}\big( B^\Omega_{12 \varepsilon}(z_0) \big)$. 
Since $W^{1,n}\big( B^\Omega_{12 \varepsilon}(z_0) \big) \hookrightarrow BMO^\infty\big( B^\Omega_{12 \varepsilon}(z_0) \big)$, we have that $p \in BMOL^2\big( B^\Omega_{12 \varepsilon}(z_0) \big)$.

\begin{lemma} \label{FBL}
Let $z_0 \in \Gamma$ and $\mathbf{R}_{h_{z_0}^\ast}^n$ be the perturbed $C^3$ half space defined in Lemma \ref{EBP} such that $B^\Omega_{12 \varepsilon}(z_0) \subset \mathbf{R}_{h_{z_0}^\ast}^n$.
There exists a constant $C = C(\alpha,\beta,K)>0$, independent of $z_0 \in \Gamma$, such that
\begin{equation} \label{FBLE}
\begin{split}
&\| \varphi_{z_0} f_0 \|_{vBMOL^2\big( \mathbf{R}_{h_{z_0}^\ast}^n \big)} + \| \varphi_{z_0} \nabla p \|_{vBMOL^2\big( \mathbf{R}_{h_{z_0}^\ast}^n \big)} \\
&\ \ \leq C \bigg( \| \varphi_{z_0} f \|_{vBMOL^2\big( \mathbf{R}_{h_{z_0}^\ast}^n \big)} + 2 \| \nabla \varphi_{z_0} (p - M_{z_0}) \|_{vBMOL^2\big( \mathbf{R}_{h_{z_0}^\ast}^n \big)} + 2 \| \omega_{z_0}^\ast \|_{vBMOL^2\big( \mathbf{R}_{h_{z_0}^\ast}^n \big)} \bigg)
\end{split}
\end{equation}
where $\omega_{z_0}^\ast$ denotes the zero extension of $\omega_{z_0}$ to $\mathbf{R}_{h_{z_0}^\ast}^n$.
\end{lemma}
\begin{proof}
By Morrey's inequality (\ref{MOR}), we see that $\omega_{z_0} = B^{z_0}_{2n}\big( \nabla \varphi_{z_0} \cdot f_0 \big) \in L^\infty\big( B^\Omega_{12 \varepsilon}(z_0) \big)^n$. 
Since $\operatorname{supp} \omega_{z_0}^\ast \subset \overline{U_{5 \varepsilon}(z_0)} \subset \overline{B^\Omega_{11 \varepsilon}(z_0)}$, obviously $\omega_{z_0}^\ast \in vBMOL^2\big( \mathbf{R}_{h_{z_0}^\ast}^n \big)$.
Let $0<r<\frac{\varepsilon}{2}$ and $z_\ast \in \Gamma$. If there exists $z \in B_r(z_\ast) \cap \Gamma$ such that $|z-z_0| \geq 12 \varepsilon$, for any other $y \in B_r(z_\ast) \cap \Gamma$ we must have that $|y-z_0| \geq |z-z_0| - |z-y| > 11 \varepsilon$. 
Hence, if there exists $z \in B_r(z_\ast)$ such that $|z-z_0| \geq 12 \varepsilon$, we see that $\varphi_{z_0} = 0$   in $B_r(z_\ast) \cap \mathbf{R}_{h_{z_0}^\ast}^n$.
If there exists $z \in B_r(z_\ast)$ such that $|z-z_0|<11 \varepsilon$, then we have that
\[
\frac{1}{|B_r(z_\ast)|} \int_{B_r(z_\ast) \cap \mathbf{R}_{h_{z_0}^\ast}^n} \big| \nabla d \cdot \big( \varphi_{z_0} f \big) \big| \, dy \leq [\nabla d \cdot f]_{b^{\frac{\varepsilon}{2}}(\Gamma)}
\]
and $\nabla d \cdot \nabla \varphi_{z_0} = 0$ in $B_r(z_\ast)$ by Proposition \ref{CFI}. 
Since a perturbed $C^3$ half space is a uniform domain, by Proposition \ref{BLRM} we see that $\varphi_{z_0} f \in BMOL^2\big( \mathbf{R}_{h_{z_0}^\ast}^n \big)^n$. 
Since $\nabla \varphi_{z_0} (p - M_{z_0}) \in BMO^{\frac{\varepsilon}{2}}\big( \mathbf{R}_{h_{z_0}^\ast}^n \big)^n \cap L^2\big( \mathbf{R}_{h_{z_0}^\ast}^n \big)^n$, we indeed have that $\nabla \varphi_{z_0} (p - M_{z_0}) \in BMOL^2\big( \mathbf{R}_{h_{z_0}^\ast}^n \big)^n$ as the $BMO^\mu$-seminorm where $\mu \geq \frac{\varepsilon}{2}$ can be estimated by the $L^2$-norm. Hence, we can follow the idea in \cite{FKS} and \cite{FKSH} to treat $\varphi_{z_0} f + \nabla \varphi_{z_0} (p - M_{z_0}) - \omega_{z_0}^\ast$ as an element of $vBMOL^2\big( \mathbf{R}_{h_{z_0}^\ast}^n \big)$.

Let us recall that there exists a constant $M_0>0$, such that for any perturbed $C^1$ half space $\mathbf{R}_w^n$ with $\| \nabla' w \|_{L^\infty(\mathbf{R}^n)} \leq M_0$ and $1<r<\infty$, the Helmholtz decomposition of $L^r\big( \mathbf{R}_w^n \big)^n$ holds, see \cite[Lemma 2.1]{FKSH}, \cite[Lemma 3.8 (a)]{SiSo}. Moreover, in this case we have that
\[
L_\sigma^r\big( \mathbf{R}_w^n \big) = \{ v \in L^r\big( \mathbf{R}_w^n \big)^n \bigm| \operatorname{div} v = 0 \;\, \text{in} \;\, \mathbf{R}_w^n, v \cdot \mathbf{n} = 0 \;\, \text{on} \;\, \partial \mathbf{R}_w^n \},
\]
which is a consequence of the fact that the $\| \cdot \|_{L^r(\mathbf{R}_w^n)}$-closure of $\nabla C^\infty_{\mathrm{c}}\big( \overline{\mathbf{R}_w^n} \big)$ equals $G^r\big( \mathbf{R}_w^n \big)$, see \cite[Lemma 3.7]{SiSo}. Hence, as $\varepsilon<\frac{M_0}{12nK}$, for any $r \in (1,\infty)$ we have that $\widetilde{L}^r\big( \mathbf{R}_{h_{z_0}^\ast}^n \big)$ admits the Helmholtz decomposition and
\[
\widetilde{L}_\sigma^r\big( \mathbf{R}_{h_{z_0}^\ast}^n \big) = \{ v \in \widetilde{L}^r\big( \mathbf{R}_{h_{z_0}^\ast}^n \big) \bigm| \operatorname{div} v = 0 \;\, \text{in} \;\, \mathbf{R}_{h_{z_0}^\ast}^n, v \cdot \mathbf{n} = 0 \;\, \text{on} \;\, \partial \mathbf{R}_{h_{z_0}^\ast}^n \}.
\]

Lemma \ref{EBP} shows that $\mathbf{R}_{h_{z_0}^\ast}^n$ is a perturbed $C^3$ half space which has small perturbation. As a result, the Helmholtz decomposition of $vBMOL^2\big( \mathbf{R}_{h_{z_0}^\ast}^n \big)$ holds \cite[Theorem 1]{GigaGuP}.
There exist $g_{0,\ast} \in vBMOL_\sigma^2\big( \mathbf{R}_{h_{z_0}^\ast}^n \big)$ and $\nabla h_\ast \in GvBMOL^2\big( \mathbf{R}_{h_{z_0}^\ast}^n \big)$ such that $\varphi_{z_0} f + \nabla \varphi_{z_0} (p - M_{z_0}) - \omega_{z_0}^\ast = g_{0,\ast} + \nabla h_\ast$ where
\begin{align*}
vBMOL_\sigma^2\big( \mathbf{R}_{h_{z_0}^\ast}^n \big) &= \{ v \in vBMOL^2\big( \mathbf{R}_{h_{z_0}^\ast}^n \big) \bigm| \operatorname{div} v = 0 \;\, \text{in} \;\, \mathbf{R}_{h_{z_0}^\ast}^n, v \cdot \mathbf{n} = 0 \;\, \text{on} \;\, \partial \mathbf{R}_{h_{z_0}^\ast}^n \}, \\
GvBMOL^2\big( \mathbf{R}_{h_{z_0}^\ast}^n \big) &= \bigg\{ \nabla h \in vBMOL^2\big( \mathbf{R}_{h_{z_0}^\ast}^n \big) \bigm| h \in \underset{r \in [1,\infty)}{\bigcap} L^r_{loc}\big( \mathbf{R}_{h_{z_0}^\ast}^n \big) \bigg\}.
\end{align*}
By Proposition \ref{ITP}, we see that $g_{0,\ast} \in \widetilde{L}_\sigma^{2n}\big( \mathbf{R}_{h_{z_0}^\ast}^n \big)$ and $\nabla h_\ast \in \widetilde{G}^{2n}\big( \mathbf{R}_{h_{z_0}^\ast}^n \big)$. Then, by the uniqueness of the Helmholtz decomposition of $\widetilde{L}^{2n}\big( \mathbf{R}_{h_{z_0}^\ast}^n \big)$, we conclude that the equality
\[
\varphi_{z_0} f + \nabla \varphi_{z_0} (p - M_{z_0}) - \omega_{z_0}^\ast = \big( \varphi_{z_0} f_0 - \omega_{z_0}^\ast \big) + \nabla \big( \varphi_{z_0} (p - M_{z_0}) \big)
\]
is indeed the Helmholtz decomposition of $\varphi_{z_0} f + \nabla \varphi_{z_0} \cdot (p - M_{z_0}) - \omega_{z_0}^\ast$ in $vBMOL^2\big( \mathbf{R}_{h_{z_0}^\ast}^n \big)$. Together with Lemma \ref{EBP}, the estimate of the Helmholtz decomposition says that
\begin{align*}
&\| \varphi_{z_0} f_0 - \omega_{z_0}^\ast \|_{vBMOL^2\big( \mathbf{R}_{h_{z_0}^\ast}^n \big)} + \| \nabla \big( \varphi_{z_0} (p - M_{z_0}) \big) \|_{vBMOL^2\big( \mathbf{R}_{h_{z_0}^\ast}^n \big)} \\
&\ \ \leq C(\alpha, \beta, K) \| \varphi_{z_0} f + \nabla \varphi_{z_0} (p - M_{z_0}) - \omega_{z_0}^\ast \|_{vBMOL^2\big( \mathbf{R}_{h_{z_0}^\ast}^n \big)}.
\end{align*}
We finally obtain Lemma \ref{FBL} by a simple triangle inequality.
\end{proof}

By considering estimate (\ref{FBLE}) for every $z_0 \in \Gamma$, we can deduce an estimate which controls the $BMO^{\frac{\varepsilon}{2}}$-estimate for $f_0$ and $\nabla p$ in $\Gamma_{3 \varepsilon}^{\mathbf{R}^n}$ together with the $b^\varepsilon$-seminorm for $\nabla d \cdot f_0$ and $\nabla d \cdot \nabla p$.

\begin{lemma} \label{FPBB}
The estimate
\begin{align} \label{BE}
[f_0]_{BMO^{\frac{\varepsilon}{2}}\big( \Gamma_{3 \varepsilon}^{\mathbf{R}^n} \big)} + [\nabla d \cdot f_0]_{b^\varepsilon(\Gamma)} + [\nabla p]_{BMO^{\frac{\varepsilon}{2}}\big( \Gamma_{3 \varepsilon}^{\mathbf{R}^n} \big)} + [\nabla d \cdot \nabla p]_{b^\varepsilon(\Gamma)} \leq \frac{C}{\varepsilon^2} \| f \|_{vBMOL^2(\Omega)}.
\end{align}
holds with $C=C(\alpha,\beta,K)>0$.
\end{lemma}
\begin{proof}
Let $x \in \Gamma_{3 \varepsilon}^{\mathbf{R}^n} \cap \Omega$ and $r \in (0,\frac{\varepsilon}{2})$ be such that $B_r(x) \subset \Gamma_{3 \varepsilon}^{\mathbf{R}^n} \cap \Omega$. Since $\varepsilon<\frac{R_\ast}{3}$, there exists a unique $z_0 \in \Gamma$ such that $|x - z_0| = d(x)$. 
For $y \in B_r(x)$, we let $y_0 \in \Gamma$ be the unique projection of $y$ on $\Gamma$.
Since $|y_0 - z_0| \leq d(y) + d(x) + r < 7 \varepsilon$, we have that $x,y \in U_{7 \varepsilon}(z_0)$. 
By estimate (\ref{UCG}), we have that $\| \nabla F_0^{-1} \|_{L^\infty\big( U_{7 \varepsilon}(z_0) \big)} < 2$ where $F_0^{-1}: U_{7 \varepsilon}(z_0) \to V_{7 \varepsilon}$ denotes the normal coordinate change in $U_{7 \varepsilon}(z_0)$. Thus, by mean value theorem we deduce that
\[
\big| y_0' - z_0' \big| \leq \big| F_0^{-1}(y) - F_0^{-1}(x) \big| \leq 2r < \varepsilon,
\]
i.e., we show that $B_r(x) \subset U_{3 \varepsilon}(z_0)$.
By Proposition \ref{CFI}, we have that $\varphi_{z_0} = 1$ in $U_{3 \varepsilon}(z_0)$.
Hence, it holds that
\begin{align*}
\frac{1}{|B_r(x)|} \int_{B_r(x)} \big| f_0 - (f_0)_{B_r(x)} \big| \, dy &\leq [f_0]_{BMO^\infty\big( U_{3 \varepsilon}(z_0) \big)} \leq \| \varphi_{z_0} f_0 \|_{vBMOL^2\big( \mathbf{R}_{h_{z_0}^\ast}^n \big)}, \\
\frac{1}{|B_r(x)|} \int_{B_r(x)} \big| \nabla p - (\nabla p)_{B_r(x)} \big| \, dy &\leq [\nabla p]_{BMO^\infty\big( U_{3 \varepsilon}(z_0) \big)} \leq \| \varphi_{z_0} \nabla p \|_{vBMOL^2\big( \mathbf{R}_{h_{z_0}^\ast}^n \big)}.
\end{align*}
In addition, let $z_0 \in \Gamma$ and $r \in (0,\varepsilon)$. For $y \in B_r(z_0) \cap \Omega$, we still let $y_0$ be the projection of $y$ on $\Gamma$. Since $|y_0 - z_0|<2 \varepsilon$, trivially $y \in U_{3 \varepsilon}(z_0)$. Hence, it holds that
\begin{align*}
r^{-n} \int_{B^\Omega_r(z_0)} \big| \nabla d \cdot f_0 \big| \, dy &= r^{-n} \int_{B^\Omega_r(z_0)} \big| \nabla d \cdot \varphi_{z_0} f_0 \big| \, dy \leq [ \varphi_{z_0} f_0 ]_{b^\varepsilon \big( \partial \mathbf{R}_{h_{z_0}^\ast}^n \big)}, \\
r^{-n} \int_{B^\Omega_r(z_0)} \big| \nabla d \cdot \nabla p \big| \, dy &= r^{-n} \int_{B^\Omega_r(z_0)} \big| \nabla d \cdot \varphi_{z_0} \nabla p \big| \, dy \leq [ \varphi_{z_0} \nabla p ]_{b^\varepsilon \big( \partial \mathbf{R}_{h_{z_0}^\ast}^n \big)}.
\end{align*}
Therefore, in order to obtain Lemma \ref{FPBB}, it is sufficient to estimate the right hand side of estimate (\ref{FBLE}).

Let $z_0 \in \Gamma$.
Since $\varepsilon< \frac{1}{8nK}$, we have that $\operatorname{supp} \varphi_{z_0} \subset \overline{U_{4 \varepsilon}(z_0)} \subset B_{10 \varepsilon}(z_0)$. Thus, if $z \in \partial \mathbf{R}_{h_{z_0}^\ast}^n$ is such that $|z - z_0| \geq 11 \varepsilon$, then for any $y \in B_\varepsilon(z)$ we have that $|y - z_0| > 10 \varepsilon$, i.e., in this case we have that $\varphi_{z_0} = 0$ in $B_\varepsilon(z)$. If $z \in \partial \mathbf{R}_{h_{z_0}^\ast}^n$ is such that $|z - z_0| < 11 \varepsilon$, then $B_\varepsilon(z) \subset B_{12 \varepsilon}(z_0)$. 
In other words, we have that
\[
[\nabla d_{\partial \mathbf{R}_{h_{z_0}^\ast}^n} \cdot \varphi_{z_0} f]_{b^\varepsilon\big( \partial \mathbf{R}_{h_{z_0}^\ast}^n \big)} \leq [\nabla d_\Gamma \cdot \varphi_{z_0} f]_{b^\varepsilon\big( \Gamma \cap B_{12 \varepsilon}(z_0) \big)} \leq [\nabla d_\Gamma \cdot f]_{b^\varepsilon(\Gamma)}, 
\]
where $d_{\partial \mathbf{R}_{h_{z_0}^\ast}^n}$ denotes the distance function for $\partial \mathbf{R}_{h_{z_0}^\ast}^n$ in $\mathbf{R}_{h_{z_0}^\ast}^n$.
Moreover, by Proposition \ref{CFI} we also have that
\[
[\nabla d_{\partial \mathbf{R}_{h_{z_0}^\ast}^n} \cdot \nabla \varphi_{z_0} (p - M_{z_0})]_{b^\varepsilon\big( \partial \mathbf{R}_{h_{z_0}^\ast}^n \big)} \leq [\nabla d_\Gamma \cdot \nabla \varphi_{z_0} (p - M_{z_0})]_{b^\varepsilon\big( \Gamma \cap B_{12 \varepsilon}(z_0) \big)} = 0.
\]

Note that for $x \in \mathbf{R}_{h_{z_0}^\ast}^n$ and $r \in (0, \frac{\varepsilon}{2})$ which satisfies $B_r(x) \cap B^\Omega_{11 \varepsilon}(z_0) \neq \emptyset$ and $B_r(x) \subset \mathbf{R}_{h_{z_0}^\ast}^n$, we must have that $B_r(x) \subset B^\Omega_{12 \varepsilon}(z_0)$. Hence, by Proposition \ref{CFI}, Proposition \ref{BLRM} and Lemma \ref{BOL}, we deduce that
\begin{align*}
&\| \varphi_{z_0} f \|_{BMOL^2\big( \mathbf{R}_{h_{z_0}^\ast}^n \big)} = \| \varphi_{z_0} f \|_{BMO^{\frac{\varepsilon}{2}}\big( \mathbf{R}_{h_{z_0}^\ast}^n \big) \cap L^2\big( \mathbf{R}_{h_{z_0}^\ast}^n \big)} = \| \varphi_{z_0} f \|_{BMOL^2\big( B^\Omega_{12 \varepsilon}(z_0) \big)} \\
&\ \ \leq C(\alpha, \beta, K) \| \varphi_{z_0} \|_{C^1\big( B^\Omega_{12 \varepsilon}(z_0) \big)} \| f \|_{BMOL^2\big( B^\Omega_{12 \varepsilon}(z_0) \big)} \leq \frac{C(\alpha, \beta, K)}{\varepsilon} \| f \|_{vBMOL^2(\Omega)}
\end{align*}
and
\begin{align*}
\| \nabla \varphi_{z_0} (p - M_{z_0}) \|_{BMOL^2\big( \mathbf{R}_{h_{z_0}^\ast}^n \big)} &= \| \nabla \varphi_{z_0} (p - M_{z_0}) \|_{BMO^{\frac{\varepsilon}{2}}\big( \mathbf{R}_{h_{z_0}^\ast}^n \big) \cap L^2\big( \mathbf{R}_{h_{z_0}^\ast}^n \big)} \\
&= \| \nabla \varphi_{z_0} (p - M_{z_0}) \|_{BMOL^2\big( B^\Omega_{12 \varepsilon}(z_0) \big)} \\
&\leq C(\alpha, \beta, K) \| \nabla \varphi_{z_0} \|_{C^1\big( B^\Omega_{12 \varepsilon}(z_0) \big)} \| p - M_{z_0} \|_{BMOL^2\big( B^\Omega_{12 \varepsilon}(z_0) \big)} \\
&\leq \frac{C(\alpha, \beta, K)}{\varepsilon^2} \| p - M_{z_0} \|_{BMOL^2\big( B^\Omega_{12 \varepsilon}(z_0) \big)}.
\end{align*}
By estimate (\ref{EWBMO}), H$\ddot{\text{o}}$lder's inequality and Poincar$\acute{\text{e}}$ inequality (\ref{POI}), we deduce that
\[
[p - M_{z_0}]_{BMO^\infty\big( B^\Omega_{12 \varepsilon}(z_0) \big)} \leq C(n) \| \nabla p \|_{L^n\big( B^\Omega_{12 \varepsilon}(z_0) \big)} \leq C(n) \| \nabla p \|_{\widetilde{L}^{2n}(\Omega)}
\]
and
\[
\| p - M_{z_0} \|_{L^2\big( B^\Omega_{12 \varepsilon}(z_0) \big)} \leq C(\alpha, \beta, K) \varepsilon^{n-1} \| \nabla p \|_{L^n\big( B^\Omega_{12 \varepsilon}(z_0) \big)} \leq C(\alpha, \beta, K) \| \nabla p \|_{\widetilde{L}^{2n}(\Omega)}.
\]
Hence, we obtain that
\[
\| \nabla \varphi_{z_0} (p - M_{z_0}) \|_{BMOL^2\big( \mathbf{R}_{h_{z_0}^\ast}^n \big)} \leq \frac{C(\alpha, \beta, K)}{\varepsilon^2} \| \nabla p \|_{\widetilde{L}^{2n}(\Omega)}.
\]
Finally, by Morrey's inequality (\ref{MOR}) and estimate (\ref{BGOz}), we see that
\begin{align*}
\| \omega_{z_0}^\ast \|_{vBMOL^2\big( \mathbf{R}_{h_{z_0}^\ast}^n \big)} &\leq C(n) \| \omega_{z_0} \|_{L^\infty\big( B^\Omega_{12 \varepsilon}(z_0) \big)} \leq \frac{C(\alpha, \beta, K)}{\varepsilon} \| \omega_{z_0} \|_{W^{1,2n}\big( B^\Omega_{12 \varepsilon}(z_0) \big)} \\
&\leq \frac{C(\alpha, \beta, K)}{\varepsilon^2} \| f_0 \|_{L^{2n}(\Omega)}.
\end{align*}
By estimate (\ref{HE2N}), we obtain Lemma \ref{FPBB}.
\end{proof}

Combine Lemma \ref{FPBI} and Lemma \ref{FPBB}, we obtain Theorem \ref{MT}.

\subsection{Characterizations of the solenoidal space and the gradient space} \label{sub:CSoGr}
Let $D \subset \mathbf{R}^n$ be either a uniformly $C^2$ domain or a uniform domain.
For $1<r<\infty$, by H$\ddot{\text{o}}$lder's inequality we can easily see that $p \in L_{loc}^r(D)$ implies that $p \in L_{loc}^1(D)$. 
Reversely, knowing that $p \in L_{loc}^1(D)$ does not necessarily mean that $p \in L_{loc}^r(D)$.
However, if we further know that $\nabla p \in BMOL^2(D)^n$, then the story is different.

\begin{lemma} \label{CGBL}
Let $D \subset \mathbf{R}^n$ be either a uniformly $C^2$ domain or a uniform domain.
Then for any $p \in L_{loc}^1(D)$ such that $\nabla p \in BMOL^2(D)^n$, we have that $p \in L_{loc}^r(D)$ for all $1<r<\infty$.
\end{lemma}
\begin{proof}
Let $Q \subset \subset D$.
We pick an arbitrary open ball $B_\ast$ such that $Q \subset B_\ast$.
Let $B^\Omega_\ast := B_\ast \cap \Omega$ and $\mathcal{I}^\ast_\varepsilon$ to be the set of all $x \in \overline{B^\Omega_\ast}$ such that $d_\Gamma(x) \geq \varepsilon$.
Note that the union
\[
\left( \underset{x \in \mathcal{I}^\ast_\varepsilon}{\bigcup} \; B_\varepsilon(x) \right) \bigcup \left( \underset{z \in \Gamma}{\bigcup} \; B_\varepsilon(z) \right)
\]
gives us an open cover for $\overline{B^\Omega_\ast}$ as
\[
\Gamma_\varepsilon^{\mathbf{R}^n} = \underset{z \in \Gamma}{\bigcup} \; B_\varepsilon(z).
\]
Hence, there exist $\{ x_i \in \mathcal{I}^\ast_\varepsilon \bigm| 1 \leq i \leq N \}$ and $\{ z_j \in \Gamma \bigm| 1 \leq j \leq M \}$ such that
\[
\overline{B^\Omega_\ast} \subset \left( \bigcup_{i=1}^N \; B_\varepsilon(x_i) \right) \bigcup \left( \bigcup_{j=1}^M \; B_\varepsilon(z_j) \right).
\]
By estimate (\ref{EWBMO}), we see that
\[
[p]_{BMO^\infty\big( B_\varepsilon(x_i) \big)} \leq C(n) \| \nabla p \|_{L^n\big( B_\varepsilon(x_i) \big)}, \quad [p]_{BMO^\infty\big( B^\Omega_\varepsilon(z_j) \big)} \leq C(n) \| \nabla p \|_{L^n\big( B^\Omega_\varepsilon(z_j) \big)}.
\] 
By Proposition \ref{ITP}, we note that $\| \nabla p \|_{L^n(D)}$ can be controlled by $\| \nabla p \|_{BMOL^2(D)}$.
Hence, we have that $p \in BMOL^1\big( B_\varepsilon(x_i) \big) \cap BMOL^1\big( B^\Omega_\varepsilon(z_j) \big)$ for any $i,j$. Using Proposition \ref{ITP} again, we deduce that 
\[
\| p \|_{L^r\big( B_\varepsilon(x_i) \big)} \leq C \| p \|_{BMOL^1\big( B_\varepsilon(x_i) \big)}, \quad \| p \|_{L^r\big( B^\Omega_\varepsilon(z_j) \big)} \leq C \| p \|_{BMOL^1\big( B^\Omega_\varepsilon(z_j) \big)}.
\]
Summing up all $i$ and $j$, we see that
\[
\| p \|_{L^r(Q)} \leq \| p \|_{L^r\big( B^\Omega_\ast \big)} \leq C \left( \sum_{i=1}^N \| p \|_{BMOL^1\big( B_\varepsilon(x_i) \big)} + \sum_{j=1}^M \| p \|_{BMOL^1\big( B^\Omega_\varepsilon(z_j) \big)} \right)<\infty.
\]
\end{proof}

By Lemma \ref{CGBL}, we see that
\begin{align*}
vBMO^{\infty,\infty}(\Omega) \cap \widetilde{G}^{2n}(\Omega) &= \{ \nabla p \in vBMOL^2(\Omega) \bigm| p \in L_{loc}^{2n}(\Omega) \} \\
&= \{ \nabla p \in vBMOL^2(\Omega) \bigm| p \in L_{loc}^1(\Omega) \}.
\end{align*}

In order to obtain a characterization of the solenoidal space $vBMO^{\infty,\infty}(\Omega) \cap \widetilde{L}^{2n}_\sigma(\Omega)$, we observe the following simple fact. 
\begin{proposition} \label{Chrac:LTq}
Let $D \subset \mathbf{R}^n$ be a uniformly $C^1$ domain. For $2 \leq q < \infty$, it holds that
\[
L_\sigma^2(D) \cap L^q(D)^n = \widetilde{L}_\sigma^q(D).
\]
\end{proposition}
\begin{proof}
For $f_\ast \in L_\sigma^2(D) \cap L^q(D)^n$, let $f_\ast = f_0 + \nabla p$ be the Helmholtz decomposition of $f_\ast$ in $\widetilde{L}^q(D)$ with $f_0 \in \widetilde{L}^q_\sigma(D)$ and $\nabla p \in \widetilde{G}^q(D)$. Note that $0 = (f_0 - f_\ast) + \nabla p$ with $f_0 - f_\ast \in L_\sigma^2(D)$ and $\nabla p \in G^2(D)$ would imply that
\[
\| f_0 - f_\ast \|_{L^2(D)}^2 = \langle f_0 - f_\ast, f_0 - f_\ast \rangle = - \langle f_0 - f_\ast, \nabla p \rangle = 0,
\] 
i.e., we must have $f_0 = f_\ast$. 
\end{proof}

Combine Proposition \ref{Chrac:LTq} with Proposition \ref{ITP}, we obtain the fact that
\[
vBMO^{\infty,\infty}(\Omega) \cap \widetilde{L}^{2n}_\sigma(\Omega) = vBMO^{\infty,\infty}(\Omega) \cap L_\sigma^2(\Omega).
\]

\section*{Acknowledgement}
The first author was partly supported by the Japan Society for the Promotion of Science through grants No. 20K20342 (Kaitaku), No. 19H00639 (Kiban A) and by Arithmer Inc., Daikin Industries Ltd. and Ebara Corporation through collaborative grants.
%

%
%



\end{document}